\documentclass[reqno,12pt,letterpaper]{amsart}
\usepackage{amsmath,amssymb,amsthm,graphicx,mathrsfs,url}
\usepackage[usenames,dvipsnames]{color}
\usepackage[colorlinks=true,linkcolor=Red,citecolor=Green]{hyperref}
\usepackage{floatrow}
\usepackage[usenames,dvipsnames]{xcolor}
\usepackage{tikz} 
\usetikzlibrary{decorations.markings}
\usepackage{slashed}
\usepackage{graphicx}
\usepackage[colorlinks=true]{hyperref}
\usepackage{color}
\usepackage[dvipsnames]{xcolor}
\usepackage{amsmath,amsfonts}
\usepackage{calrsfs}
\usepackage{amsmath,environ}
\usepackage{floatrow}
\usepackage{autonum}
\usepackage{amsmath, amsthm}
\usepackage{dsfont}
\usepackage{fancybox}
\usepackage{inputenc}
\usetikzlibrary{decorations.pathreplacing}
\usepackage[usenames]{color}
\usepackage{url}
\usepackage{enumitem}
\usepackage{pgf,tikz}
\usetikzlibrary{arrows}
\usetikzlibrary{calc}

\usepackage{caption}
\usepackage{subcaption}
\DeclareMathAlphabet{\pazocal}{OMS}{zplm}{m}{n}

\numberwithin{equation}{section}

\newcommand{\logif}{\text{ \textbf{if} }}

\newcommand{\din}{d^{-}} 
\newcommand{\dout}{d^{+}} 
\newcommand{\Lin}{L^{-}} 
\newcommand{\Lout}{L^{+}}
\newcommand{\Lio}{L^{\pm}}

\newcommand{\IG}{I_{\ell^2(G)}}

\makeatletter
\DeclareRobustCommand{\cev}[1]{%
	\mathpalette\do@cev{#1}%
}
\newcommand{\do@cev}[2]{%
	\fix@cev{#1}{+}%
	\reflectbox{$\m@th#1\vec{\reflectbox{$\fix@cev{#1}{-}\m@th#1#2\fix@cev{#1}{+}$}}$}%
	\fix@cev{#1}{-}%
}
\newcommand{\fix@cev}[2]{%
	\ifx#1\displaystyle
	\mkern#23mu
	\else
	\ifx#1\textstyle
	\mkern#23mu
	\else
	\ifx#1\scriptstyle
	\mkern#22mu
	\else
	\mkern#22mu
	\fi
	\fi
	\fi
}

\makeatother

\newcommand{\dist}{{\operatorname{dist}}}

\renewcommand{\Re}{\operatorname{Re}}

\newtheoremstyle{definition_style}
{15pt}
{10pt}
{\itshape}
{}
{\bfseries}
{ }
{\newline}
{}
\theoremstyle{definition}
\newtheorem{Th}{Theorem}[section]

\newtheorem{Prop}[Th]{Proposition}

\newtheorem{Def}[Th]{Definition}
\newtheorem{Thm}[Th]{Theorem}

\newtheorem{Lem}[Th]{Lemma}
\newtheorem{Rem}[Th]{Remark}

\newtheorem{Ex}[Th]{Example}

\newtheorem{Cor}[Th]{Corollary}
\newtheorem{Ass}[Th]{Assumption}

\title[Effective Graph Laplacians and resolvent convergence
]{Di-graphs with tightly connected clusters: Effective 
	Graph 
	Laplacians and resolvent convergence
}

\author{Christian Koke}

\NewEnviron{equations}{%
	\begin{equation}\begin{gathered}
			\BODY
	\end{gathered}\end{equation}
}

\newtheorem{prop}{Proposition}[section]

\begin{document}
	
	\vspace*{-5mm}
	
	\maketitle
	
	\begin{abstract}
		In this note, we study Laplacians on graphs for which connectivity within certain sub-graphs tends to infinity. Our main focus 
		are graphs 	sharing
				a common node set on which
		  edge weights within certain clusters grow to infinity.
		As intra-cluster connectivity increases, we show that the corresponding graph Laplacians converge—in the resolvent sense—to an effective graph Laplacian. This effective limit Laplacian is defined on a coarsened graph, where each highly connected cluster is collapsed into a single node.
		In the undirected setting, the effective Laplacian arises naturally from agregating over tightly connected clusters.
		  In the directed case, the limiting graph structure depends on the precise manner in which connectivity increases; with the corresponding effects mediated by the left and right kernel structure of the Laplacian restricted to high-connectivity clusters.
		Our results shed light on the emergence of coarse-grained dynamics in large-scale networks and contribute to spectral graph theory of directed graphs.
	\end{abstract}
	
	\section{Introduction}\label{intro}
Graphs are a fundamental tool for modeling pairwise interactions in systems across a wide range of disciplines, including biology \cite{barabasi2000large}, physics \cite{watts1998collective}, computer science \cite{cormen2009introduction}, as well as social sciences \cite{wasserman1994social}.
Fundamentally graphs may be distinguished by whether they are \emph{undirected} or \emph{directed}.
While directed graphs are mathematically somewhat more subtle and less well understood than their undirected counterparts \cite{Veerman2017RandomWO,Veerman2020APO}, they are essential for modeling a diverse array of systems. Applications include the structure of the internet \cite{broder2000graph}, social networks \cite{carrington2005models}, food webs \cite{may1973qualitative}, epidemic spread \cite{jombert2011reconstructing}, chemical reaction networks \cite{rao2013graph}, graph-structured databases \cite{angles2008survey}, communication networks \cite{ahlswede2000network}, and coordination in multi-agent systems \cite{fax2004information}.

In many such settings, tightly connected sub graphs -- often referred to as communities or clusters -- emerge and play a central role in shaping the system’s global behavior. Detecting such communities has been shown to be essential for uncovering functional subunits and understanding how local organization drives global dynamics \cite{fortunato2007resolution, newman2006modularity, fortunato2010community, nematzadeh2014optimal, onella2007structure}.

.

Given the prevalence of such hierarchical structure, it 
is natural to investigate
representations of networks in which tightly connected subgraphs are abstracted as single nodes. This process, commonly referred to as coarse-graining \cite{Loukas2019, Safro2015}, not only reduces the complexity of the system but also reveals its large-scale organization \cite{Villegas2023, Poggialini2025},
facilitates scalable computation \cite{10.24963/ijcai.2023/752, 10.1145/3447548.3467256} and provides insight into the mechanisms by which local interactions influence global behavior \cite{PhysRevResearch.4.033196, PhysRevResearch.7.013065}.
In this present paper  we then rigorously address the question: \emph{What is the appropriate effective representation of a network containing such strongly connected clusters?}

\medskip

 To introduce our approach, consider as a motivating example a graph modeling an ensemble of coupled heat reservoirs. From the perspective of heat flow, edge weights then correspond to thermal conductivities. In the limit where the conductivity between two nodes tends to infinity, heat exchange becomes instantaneous, and the nodes equilibrate immediately---effectively behaving as a single unit. 
 
 \medskip
 
Heat flow over a graph is governed by the diffusion equation
$
	\partial_t \psi(t) = L \psi(t)
$, with graph Laplacian $L$ and global-in-time solution $\psi(t) = e^{-t L} \psi_0$. Heuristically, we expect this fast-equilibration behavior to be reflected  in the evolution operator (i.e. the heat kernel): If two nodes are tightly coupled by a very large edge weight '$\beta$', then we expect
$
	e^{-t L_\beta} \rightarrow J_\uparrow e^{-t \underline{L}} J_\downarrow
$ as $\beta \rightarrow \infty$. Here $L_\beta$  is the gaph Laplacian on the original graph.
 $\underline{L}$ is the Laplacian of a coarsened graph in which the tightly connected nodes are collapsed into a single ``supernode''. $J_\downarrow$ and $J_\uparrow$ interpolate between original and reduced graph.

\medskip

In this work, we formalize this intuition by studying the \emph{asymptotic behavior of graph Laplacians} associated with a sequence of weighted graphs $G_\beta$  in which connectivity within specified clusters increases without bound. We show that, under suitable conditions, this sequence of Laplacians converges to a limiting operator: the \emph{effective Laplacian} on a coarsened graph. The structure of this coarsified graph $\underline{G}$ is determined by the original graph $G$ and the 
spectral projection onto the kernel of the restriction of the original graph Laplacian to the tightly knit clusters.
 Instead of establishing convergence of heat-kernels directly, we find it more convenient to instead establish
 \begin{align}\label{res_conv_intro}
 		(L_\beta - z I_{\ell^2(G)})^{-1} \rightarrow  J^\uparrow  (\underline{L} - z I_{\ell^2(\underline{G})})^{-1} J^\downarrow, \quad \text{as $\beta \rightarrow \infty$};
\end{align}
i.e.  convergence of \emph{resolvents}.
This then also implies convergence of heat kernels.

\medskip
 In the  reduced graph $\underline{G}$, each high-connectivity cluster is represented as a single node. Inter-cluster connections on the original $G$ are then aggregated to determine the graph structure of $\underline{G}$ for which the convergence (\ref{res_conv_intro}) holds. The way this aggregation proceeds -- and hence also the resulting graph structure on $\underline{G}$ --
  differs between the undirected and directed settings: 
In the \emph{undirected setting}, we find that the limit graph structure is
insensitive to the specific way in which connectivity within subclusters diverges. 
In the \emph{directed case}, however, the situation is subtler: The structure of the effective graph $\underline{G}$ depends not only on the clusters themselves, but on  \emph{how} exactly connectivity is taken to infinity. This dependence is mediated by the \emph{left and right kernel structures} of the Laplacian restricted to each cluster. These kernel components encode asymmetries and directional flow patterns, introducing a rich variety of possible limiting graphs even for the same cluster partitioning. 

\medskip

Taken together, these results provide a rigorous framework for understanding coarse-graining and effective descriptions for potentially directed networks,
and contribute to the broader development of spectral graph theory in potentially directed and weighted settings.
Our results also have direct implications for spectral methods in clustering, 
community detection and graph based machine learning \cite{GraphScatteringBeyond, limitless, koke2024holonets}.\footnote{Additional in-progress work using methods developed in this current paper have presented at non-archival venues \cite{koke2023resolvnet, koke2024transferability, koke2025multiscale, koke2025graphscale, koke2025incorporating}}

Our paper is organized as follows: Section \ref{background} discusses the general parlance in which our results are formulated. In particular it recalls the notion of in- and out-degree Laplacians and characterizes the  kernels of these operators in terms of the connectivity structure of the corresponding directed graphs. Section \ref{undirected_convergence} then covers convergence results for undirected graphs.  Section \ref{directed_convergence} discusses convergence in the technically more involved directed setting.

	\section{Graphs and Graph Laplacians}\label{background}
	\subsection{Graphs and the notion of Connectedness}\label{graphs_and_connectedness}

The principal objects of our study are weighted and generically directed graphs \cite{anne2018sectoriality,balti2017, balti2017a, Huang2013SelfAdjoint, KostenkoNicolussi2023}):
	
	\begin{Def}
	A (weighted and directed) graph $G$  is a tuple $G = (V, m, E, a)$ where 
		\begin{enumerate}[label=(\roman*), topsep=0pt, itemsep=-1em]	
			\item $V$ is an at most countably-infinite set of vertices,\\
			\item $m: V \rightarrow (0, \infty)$ is a function assigning a positive mass $m(v) > 0$ to each  $v \in V$, \\
			\item 	$E \subseteq V \times V$ is the (locally finite) set  of edges in the graph $G$,\\
	\item  $a: V \times V \rightarrow [0,\infty)$ is a weight function satisfying $a(v,w) > 0$ iff $(v,w) \in E$.
	\end{enumerate}
The transpose $G^\intercal$ of an original graph $G$ is the tuple $G^\intercal = (V, m, E^{\intercal}, a^\intercal)$, with $E^{\intercal} = \{(i,j) \in V \times V: (j,i) \in E\}$ and $a^\intercal(i,j) := a(j,i)$.
A graph $G$ is \textit{undirected}, if $G = G^{\intercal}$.
		\end{Def}
Given a graph with edge set $E$, we will interpret the edge  $(i,j) \in E$ as originating at node $j$ and terminating at node $i$.\footnote{In contrast, e.g. \cite{Veerman2020APO} read edges 'from left to right'; interreting the edge  $(i,j)$  as originating at  $i$.}
We denote by '$i\rightsquigarrow_E j$' the fact that there exists a directed path $p = \{(i,v_{n-1}),(v_{n-1},v_{n-2}),..., (v_1,j)\} \subseteq E$ from $j$ to $i$, along which the weight function $a$ never vanishes. Informally, we say that  node $i$ ``sees"  node $j$.

\medskip

In an \emph{undirected} graph,
 this notion
 is reflexive: If $i \rightsquigarrow_E j$ then clearly also  $j \rightsquigarrow_E i$ simply by traversing the original path
 backwards.
 In the undirected setting, it is thus straightforward to define a notion of connectedness: A graph is connected if all nodes see each other. Similarly the set of connected components is determined, simply by taking the quotient with respect to the equivalence relation induced by $i \rightsquigarrow_E j$:

\begin{Def}\label{connected_components}
	The set of connected components of  an \emph{undirected} graph $G = (V, m, E, a)$
is defined as	$\mathfrak{C}_{E} = V/ \sim_{E}$. 
\end{Def}
It should be noted that the quotient space $\mathfrak{C}_{\hat{E}}$ indeed only depends on the edge set $E$, and not on the actual values $a(i,j) > 0$ takes on these edges. We will hence simply consider $\mathfrak{C}$ as the set of connected components of the graph $(V, E)$; suppressing the edge weight function in our notation.
Below, we will -- for certain graphs -- also write $\underline{V}:= \mathfrak{C}_{E} $ to denote this set of connected components of $(V, E)$.

\medskip

In the \emph{directed} setting, reflexivity of the statement '$j$ sees $i$'
  is lost. Thus '$i \rightsquigarrow_E j$' no longer induces an equivalence relation and the notion of 'connected component' as in Definition \ref{connected_components} becomes vacuous. Here, the notion of \emph{reaches} instead supersedes the concept of connected components \cite{caughveer,veerkumm, Veerman2020APO}:

\begin{Def} \label{def:reaches}
Given a graph $G$, we make the following standard definitions:
	\begin{enumerate}[label=(\roman*), topsep=0pt, itemsep=-1em]	
		\item Let \( i \in V \). The \emph{reachable set} \( R(i) \) is the set of all vertices \( j \in V \) such that there exists a path \( i \rightsquigarrow_E j \). A \emph{reach} \( R \) is a maximal reachable set; i.e. a maximal set $R$ satifying the condition 'there exists at least one node from which all other nodes in $R$ may be reached'.	\\
		\item A \emph{cabal} \( B \subseteq R \) is a subset of vertices from which all other vertices in the reach \( R \) are reachable.\footnote{Note that every reach has exactly one non-empty cabal \cite{Veerman2020APO}.}\\ 
		\item The \emph{exclusive part} \( H \subseteq R \) consists of those vertices in \( R \) that do see nodes of any other reach.\\
		\item The \emph{common part} \( C \subseteq R \) consists of those vertices in \( R \) also see vertices in other reaches.
	\end{enumerate} 
By maximality, 'being in the same reach' is an equivalence relation, determined by the set of edges $E$.  We denote this relation by '$\sim_{{E}}$' and write  $\mathfrak{R}_{{E}} := V/\sim_{E}$ for the set of reaches.
 For certain graphs, we will also use the notation $\underline{V}:= \mathfrak{R}_{E} $ below.
\end{Def}

\medskip

In the undirected setting, the notion of reaches $\mathfrak{R} = \{R\}$ clearly coincides with the concept of connected components $\mathfrak{C} = \{C\}$, so that it is indeed possible to unambiguously use the notation  '$\sim_E$' for either equivalence relation.
As important point of caution, it should however be noted that in the general directed setting, neither the total number of reaches nor the specific reaches themselves are stable under orientation-reversion ($G \mapsto G^{\intercal}$); as evident from Examples \ref{weight_vector_example_writing} \& \ref{weight_vector_example_writing_II} below, and discussed in detail e.g.  in \cite[Section 2]{Veerman2020APO}.

\subsection{Associated Hilbert spaces and Laplacians }
On the graph $G$,  we introduce the space of functions $f: V \rightarrow \mathds{C}$ square summable with respect to the measure $m$ as
\begin{align}
	\ell^2(G) = \left\{ f: V \rightarrow \mathds{C}: \sum\limits_{v \in V}  |f(v)|^2 m(v) < \infty      \right\}.
\end{align}
This is a Hilbert space, with inner product given by
$
	\langle f,g\rangle_{\ell^2(V)} = \sum_{x \in V} \overline{f(x)} g(x) m(x)
	$.
In the generically directed setting, we then introduce the in- and out-degree Laplacians on $G$ via their respective actions on elements $f \in \ell^2(V)$:

\begin{Def}\label{laplacian_definition}
	The \emph{in}-degree ($\Lin$) and \emph{out}-degree ($\Lout$) graph Laplacians are:
	\begin{align}
		[\Lin f](x) &= \frac{1}{m(x)} \left(\din(x)  f(x)   - \sum\limits_{y \in V} a(x,y) f(y) \right),\\
		[\Lout f](x) &= \frac{1}{m(x)} \left(\dout(x)  f(x)   - \sum\limits_{y \in V} a(x,y) f(y) \right).
	\end{align}
\end{Def}
Here we have made use of the in-degree $d^-$ and out-degree $d^+$ as:
\begin{align}
	\din(v) = \sum\limits_{i \in V} a(v,i) , \quad 	\dout(v) = \sum\limits_{j \in V} a(j,v).
\end{align} 
We will here focus on the setting where these operators of Definition \ref{laplacian_definition} are indeed defined on all of $\ell^2(G)$.
To ensure the necessary boundedness, we make the following assumption on interplay of degree and node-mass throughout the paper:
\begin{Ass}\label{well_behavedness_assumption}
	In- and out-degrees are uniformly bounded relative to node-mass:
		\begin{align}
			\exists C > 0 :  \quad	\din(v)	\leq C m(v) \quad \text{and} \quad \dout(v) \leq C m(v), \quad \forall v \in V .
		\end{align} 
\end{Ass}

This assumption generalizes standard boundedness characterizations in the undirected setting \cite{KostenkoNicolussi2023} and indeed ensures that the operators of Definition \ref{laplacian_definition} are bounded:

\begin{Prop}\label{tame_Laplacians}
With $C$ as in Assumption \ref{well_behavedness_assumption}, we have 	$\|\Lin\|, \|\Lout\| \leq 2C$. 
\end{Prop}
\begin{proof}
We will conduct the proof for $\Lin$; the corresponding arguments for the out-degree Laplacian proceed in complete analogy. We can split $\Lin$ as $\Lin = D^{-}-A$ into its diagonal and off-diagonal action. For the action of $D^{-}$, we find:
\begin{equation}
	\|D^{-} f\|^2 = \sum\limits_{x \in V} \frac{1}{m(x)} [\din(x)]^2 |f(x)|^2 \leq C^2   \sum\limits_{x \in V}   |f(x)|^2 m(x)  \leq C^2 \|f\|^2. 
\end{equation}
The remaining off-diagonal operator $A$ acts on $\ell^2(G)$ as an integral operator with kernel 
\begin{align}
	k(x,y) = \frac{a(x,y)}{m(x) m(y)}.
\end{align}
The Schur test (c.f. e.g. \cite[Theorem 5.2]{HalmosSunder1978}) allows us to bound $\|A\| \leq \sqrt{K_1K_2}$, with 
\begin{align}
K_1 &=	\sup\limits_{x \in V} \int_V |k(x,y)| dm(y)= \sup\limits_{x \in V} \left( \sum\limits_{y \in V} \left| \frac{a(x,y)}{m(x) m(y)} \right| m(y) \right) = 
	\sup\limits_{x \in V} \left( \frac{\din(x)}{m(x)} \right),\\
K_2 &=	\sup\limits_{y \in V} \int_V |k(x,y)| dm(x)= \sup\limits_{y \in V} \left( \sum\limits_{x \in V} \left| \frac{a(x,y)}{m(x) m(y)}\right|  m(x) \right) = 
	\sup\limits_{y \in V} \left( \frac{\dout(y)}{m(y)} \right).
\end{align}
Since both $K_1,K_2 \leq C$, we thus find $A \leq \sqrt{C^2} = C$. By a triangle inequality argument it then follows that $\|\Lin\| \leq 2C$. 
\end{proof}

The notions of in- and out-degree Laplacian clearly coincide if $G$
 is undirected. Moreover, since Proposition \ref{tame_Laplacians} guarantees boundedness, it is not hard to verify that if the underlying graph is undirected, \emph{the} graph Laplacian $L = \Lin = \Lout$ then is self-adjoint and positive semi-definite. In this setting, the Laplacian is then canonically associated to the Dirichlet-form $\mathcal{E}_G$ determined by $G$ \cite{KostenkoNicolussi2023}:
\begin{Def}
	The Dirichlet form of an undirected graph $G$ is defined as 
	\begin{align}\label{dirichlet_form}
		\mathcal{E}_G(f,f) :=  \sum\limits_{(i,j) \in E} a(i,j) |f(i) - f(j)|^2 = \langle f, L f\rangle_{\ell^2(G)}.
	\end{align}
\end{Def}

\subsection{Kernel structure of graph Laplacians  }
\label{chap:kernels}
We are interested in the effective Laplacian describing a graph in the limit of very large connectivity within certain clusters. As we will see in Section \ref{directed_convergence} this reduced oeprator will crucially depend on the kernel structure of the orginal Laplacian restricted to these tightly knit clusters. We will here thus briefly charaterize  the structure of these kernels.

\subsubsection{Kernels of Laplacians on undirected Graphs}
In the undirected setting, $L$ is self-adjoint, which makes the study of its kernel straightforward. With $\mathfrak{C}$ the set of connected components of a graph $G$,  $\chi_C$ the characteristic function of the connected component $C \subseteq V$,  and the cluster mass $\underline{m}(C) = \sum_{v \in C} m(v)$, we have the following:
\begin{prop}
	The kernel of $L$ on $G = (V, m, E, a)$ is given as 
	\begin{align}
	\ker(L) = \overline{\operatorname{span} \left\{ \chi_C : C \in \mathfrak{C} \text{ and } \underline{m}(C) < \infty \right\}}.
	\end{align}
\end{prop}
\begin{proof}
Clearly $\operatorname{span} \left\{ \chi_C : C \in \mathfrak{C} \text{ and } m(C) < \infty \right\} \subseteq \ker(L)$ and $\ker(L)$ is closed, so that "$\supseteq
$"  is proved.
To also prove "$\subseteq$", let now $f \in \ker(L)$ be arbitrary.  Since $f \in \ker(L)$,  clearly also $\langle f, L f \rangle_{\ell^2(G)} = 0$, and hence the Dirichlet form (\ref{dirichlet_form}) vanishes on $f$. Thus we see, that a necessary condition for $f$ to be in the kernel of $L$ is, that $f$ is constant on each connected component of $G$, as otherwise $\langle f, L f \rangle_{\ell^2(G)} = \mathcal{E}_G(f,f) > 0$. Suppose now there is a connected component $C$ of $G$ on which the restriction $f{\restriction_C} = \lambda \neq 0$ does not vanish. Then $|\lambda|^2 \cdot m(C) \leq \|f\|^2_{\ell^2(G)} < \infty$ and hence the mass of this connected component is finite.
\end{proof}
Since $L$ is self-adjoint, the left-kernel of $L$ is simply the transpose of the right kernel of $L$ (i.e. the image 
of the right-kernel 
under the Riesz isomorphism).
Moreover for finite graphs, the algebraic and geometric multiplicities of  $0 \in \sigma(L)$ clearly coincide and the dimension of the corresponding kernel of $L$ is equal to the number of connected components of $G$  \cite{chung1997spectral}.
\medskip

\subsubsection{Left- and Right-Kernels of Laplacians on finite directed graphs}\label{directed_kernels}
In contrast to the well established undirected setting, spectral graph theory for directed graphs  is currently still in development \cite{ Veerman2017RandomWO, Veerman2020APO, chung2005laplacians, li2012digraph}. 
 As we will see in Section \ref{directed_convergence} below, it will be sufficient for our purposes to merely understand the kernel structure of \emph{finite} directed graph Laplacians. For the in-degree Laplacian on graphs without node-weights ($m \equiv 1$) the kernel structure has recently been investigated in  \cite{Veerman2017RandomWO,Veerman2020APO}.
 Here we will thus review and extend these results. 
 We may already note
that since in $L^\pm$ is no longer self-adjoint, its right- and left- kernels generically no longer agree. Furthermire, while now the notion of \emph{reaches} suprsedes the concept of connected components, the relationship between connectivity and kernel structure becomes somwehat more involved: While  dimension and structure of $\ker(\Lio)$ are still tied to the reaches of $G$, the exact  relationship is  somewhat more subtle and dependent on whether in- or out-degree Laplacians are considered.
We will thus  consider 
these two
settings separately:

\medskip

\noindent \textbf{In-degree Laplacian $\Lin$:} For the right kernel of the in-degree Laplacian, we have:
\begin{Prop}\label{right_kernel_Lin}\cite{Veerman2020APO}
	Let $G$ be a finite digraph with reaches $ \mathfrak{R} = \{R\}$. 
	A basis of the  \emph{right} kernel of an in-degree Laplacian $\Lin$ is given by 
	 $\{\vec \gamma^{-}_{R}\}_{R \in \mathfrak{R}} \subseteq  \ell^2(G)$ , where:
	\begin{align}
		\left\{\begin{matrix}
			\vec	\gamma^{-}_{R}(j)  =1 & \logif & j\in H(R) &\textrm{(exclusive)}\\
			\vec	\gamma^{-}_{R}(j) \in (0,1) & \logif & j\in C(R) &\textrm{(common)}\\
			\vec	\gamma^{-}_{R}(j)=0 & \logif & j\not\in R &\textrm{(not in reach)}\\
		\end{matrix} \right.
	\end{align}
The values $\vec	\gamma^{-}_{R}(j) \in (0,1) $ for $j$ in the common part $C(R)$ may be calculated as 
$
		\vec{\gamma}^{-}_{R}{\restriction_{C(R)}} =   [\chi_{C(R)}\cdot\Lin{\restriction_{C(R)}} ]^{-1} [\chi_{C(R)}\cdot\Lin\chi_{H(R)}].
		$
Furthermore for all $j \in V$ arbitrary, we have $\sum_{R \in \mathfrak{R}}\,	\vec\gamma^{-}_{R}(j) = 1$. 
	\end{Prop}
The fact that $\sum_{R \in \mathfrak{R}}\,	\vec\gamma^{-}_{R}(j) = 1$  exactly ensures that the elements of the basis $\{\vec \gamma^{-}_{R}\}_{R \in \mathfrak{R}}$ constitute a (non-negative) resolution of the identity in $\ell^2(G)$.  Moreover, we note that $[\chi_{C(R)}\cdot\Lin\chi_{H(R)}]$ arises from applying $\Lin$ to the characteristic function of the exclusive part $H(R)$ and subsequently restricting the result to the common part $C(R)$ via  scalar ('entry-wise') multiplication with the characteristic function of the common part $C(R)$. The operator $[\chi_{C(R)}\cdot\Lin{\restriction_{C(R)}} ]$ constitutes the restriction of $\Lin$ to elements supported on the common part $C(R)$, post-composed with an orthogonal projection onto the common part again. It is invertible by \cite{Veerman2020APO}[Theorem 5.1].
\begin{proof}[Proof of Proposition \ref{right_kernel_Lin}]
For the case of node weights constant to one ($m \equiv 1$), this was proved in \cite[Theorem 5.1]{Veerman2020APO}. Varying node-weights however does not modify the right kernel of $\Lin$, as is evident from Definition \ref{laplacian_definition}. Hence the corresponding results also apply here.
	\end{proof}

	To proceed further and characterize the left kernel, we recall that a directed tree, is a directed graph where each vertex has at most one incoming edge, and the underlying undirected graph is both acyclic and connected (i.e., it forms a tree). When tracing the edges backwards from any vertex, we eventually reach the same vertex, known as the root. If the connectivity condition is removed, we obtain a directed forest, which is simply a disjoint union of directed trees. A \emph{spanning tree} of a (sub-)graph is a tree which visits each node in this (sub-)graph exactly once. Derived from this, is the notion of the \emph{weight-vector} of  a reach, which below we will use give a characterization of the \emph{left}-kernel structure in the in-degree setting:

	\begin{Def}\label{weight_vector} Let $\mathcal{T}^R_i$ be the set of all spanning trees of the reach $R$ that are rooted at node $i\in R$. Let  $\tau_i$ be such a spanning tree beginning at node $i$. We define the weight-vector of the reach $R$ as 
		$
		w_R(i) = \sum\limits_{\tau_i \in \mathcal{T}^R_i} \prod\limits_{(x,y)\in\tau_i} a(x,y).
		$
	\end{Def}

\begin{Ex}\label{weight_vector_example_writing}
As an example, let us consider the networks of Fig.~\ref{weight_vector_example}:
The graph $G_1$ has a single reach, with a sole spanning tree, which is centered at node '$a$'.  
For the weight vector $\omega_R$, we find $\omega_R(a) = \alpha \cdot \gamma \cdot \delta$ and 
$\omega_R(b) = \omega_R(c) = \omega_R(d) = 0$. 
Graph $G_2$ has two reaches, which we call $R = \{a,b\}$ and $S = \{c,d\}$. 
Each reach has one spanning tree, rooted at nodes $a$ and $c$ respectively. 
For the corresponding weight vectors, we thus find 
$\omega_R(a) = \alpha$, 
$\omega_R(b) = \omega_R(c) = \omega_R(d) = 0$, and 
$\omega_S(c) = \delta$, 
$\omega_S(a) = \omega_S(b) = \omega_S(d) = 0$.
Finally, graph $G_3$ has one reach and three spanning trees; all rooted at node '$a$'. 
Thus we find
$
\omega_R(a) = \alpha \cdot \eta \cdot \gamma  
+ \alpha \cdot \gamma \cdot \delta    
+ \alpha \cdot \rho \cdot \delta, \quad 
\omega_R(b) = \omega_R(c) = \omega_R(d) = 0.
$
\end{Ex}

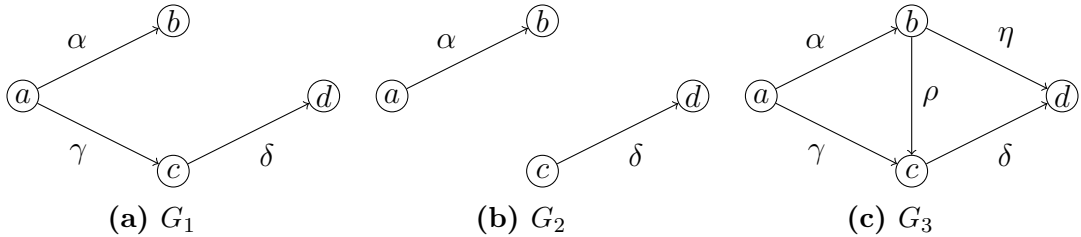
\begin{figure}[H]
	\centering
	\begin{subfigure}[b]{0.25\textwidth}
		\centering
		\begin{tikzpicture}
			\node[draw, circle, minimum size=12pt, inner sep=0pt] (a) at (0,1) {$a$};
			\node[draw, circle, minimum size=12pt, inner sep=0pt] (b) at (2,2) {$b$};
			\node[draw, circle, minimum size=12pt, inner sep=0pt] (c) at (2,0) {$c$};
			\node[draw, circle, minimum size=12pt, inner sep=0pt] (d) at (4,1) {$d$};
			
			\draw[->] (a) -- node[above left] {$\alpha$} (b);
			\draw[->] (a) -- node[below left] {$\gamma$} (c);
			\draw[->] (c) -- node[below right] {$\delta$} (d);
		\end{tikzpicture}
		\caption{$G_1$}
	\end{subfigure}
	\hspace{0.05\textwidth}
	\begin{subfigure}[b]{0.25\textwidth}
		\centering
		\begin{tikzpicture}
			\node[draw, circle, minimum size=12pt, inner sep=0pt] (a) at (0,1) {$a$};
			\node[draw, circle, minimum size=12pt, inner sep=0pt] (b) at (2,2) {$b$};
			\node[draw, circle, minimum size=12pt, inner sep=0pt] (c) at (2,0) {$c$};
			\node[draw, circle, minimum size=12pt, inner sep=0pt] (d) at (4,1) {$d$};
			
			\draw[->] (a) -- node[above left] {$\alpha$} (b);
			\draw[->] (c) -- node[below right] {$\delta$} (d);
		\end{tikzpicture}
		\caption{$G_2$}
	\end{subfigure}
	\hspace{0.05\textwidth}
	\begin{subfigure}[b]{0.25\textwidth}
		\centering
		\begin{tikzpicture}
			\node[draw, circle, minimum size=12pt, inner sep=0pt] (a) at (0,1) {$a$};
			\node[draw, circle, minimum size=12pt, inner sep=0pt] (b) at (2,2) {$b$};
			\node[draw, circle, minimum size=12pt, inner sep=0pt] (c) at (2,0) {$c$};
			\node[draw, circle, minimum size=12pt, inner sep=0pt] (d) at (4,1) {$d$};
			
			\draw[->] (a) -- node[above left] {$\alpha$} (b);
			\draw[->] (a) -- node[below left] {$\gamma$} (c);
			\draw[->] (c) -- node[below right] {$\delta$} (d);
			\draw[->] (b) -- node[right] {$\rho$} (c);
			\draw[->] (b) -- node[above right] {$\eta$} (d);
		\end{tikzpicture}
		\caption{$G_3$}
	\end{subfigure}
	\caption{Example graphs with different reaches on the same node set}
	\label{weight_vector_example}
\end{figure}

In \cite[Theorem 1]{sahi2014harmonic}, it was proved that the \emph{weight vector of an entire graph} is harmonic\footnote{In \cite{sahi2014harmonic} 'harmonic' was taken to mean that a vector is in the right-kernel of a certain out-degree Laplacian, which in our setting corresponds to the transpose $[\Lin]^\intercal$; c.f. also  Prop. \ref{in_out_transpose_prop} below.}. Here, we extend this result, and completely characterize the left-kernel of the in-degree Laplacian $\Lin$ in terms of \emph{weight vectors of the reaches} of $G$. We will use the notation  $e_j = \delta_{j} = \chi_{\{j\}}$ for the characeristic function of the set $\{j\}$:

\begin{Prop}\label{prop:leftkernel_-}
	A basis of the \emph{left} kernel of an in-degree Laplacian $\Lin$ is 
	given by the subset
	$\{\langle \cev\gamma^{-}_{R},\cdot \rangle_{\ell^2(G)}\}_{R \in \mathfrak{R}}$ of the dual space $[\ell^2(G)]^* \cong \ell^2(G)$ indexed by the reaches $\mathfrak{R}$ of $G$. Each basis element corresponding to a reach is parallel to the weight vector of said reach (i.e. $\cev \gamma^{-}_{R} \parallel \omega_R$). Furthermore we may choose a normalization so that  $\sum_{j}\, \langle \cev \gamma^{-}_{R}, e_j \rangle =1 $ with  $	\langle \cev \gamma^{-}_{R}, e_j \rangle >0 $ iff $j \in B(R)$,  $	\langle \cev \gamma^{-}_{R}, e_j \rangle  = 0 $ if $j \notin B(R)$ 
	and we have 
	the  interplay $
	\langle \cev\gamma^{-}_{R}, \vec\gamma^{-}_{S}\rangle_{\ell^2(G)} = \delta_{RS}$ with the basis $\{\vec \gamma^{-}_{R}\}_{R \in \mathfrak{R}} \subseteq  \ell^2(G)$ of the right-kernel of $\Lin$, as characterized by Proposition \ref{right_kernel_Lin} above.
\end{Prop}

It should be noted that parallelity with the weight vector $\omega_R$ already implies that $	\langle \cev \gamma^{-}_{R}, e_j \rangle \neq 0 $ if and only if $j \in B(R)$ (i.e. $j$ is in the cabal of  $R$). For all other $j \in V \setminus  B(R)$ we have $	\langle \cev \gamma^{-}_{R}, e_j \rangle = 0 $ instead, since by definition of
 'cabal' (c.f. Definition \ref{def:reaches}) there are no spanning trees of reach '$R$' that are rooted at any vertex outside the cabal $B(R)$ of this reach. 
\begin{proof}[Proof of Proposition \ref{prop:leftkernel_-}]
By \cite[Theorem 4.6]{Veerman2020APO} the algebraic and geometric multiplicities of the eigenvalue $0 \in \sigma(\Lin)$ is equal to the number of reaches   $|\mathfrak{R}|$ in the graph $G$. Next we note that the $\{\omega_R\}_{R \in \mathfrak{R}}$ have pairwise disjoint support and are thus linearly independent. Hence to show that $\{\omega_R\}_{R \in \mathfrak{R}}$ provides a basis of the left kernel of $\Lin$, it suffices to prove that each $\omega_R$ is harmonic in the sense that $\langle \omega_R, \Lin \cdot \rangle_{\ell^2(G)} = 0$.

To establish this, we follow the proof strategy of \cite[Theorem 1]{sahi2014harmonic} and first recall that a \textit{dangle} is defined as a directed graph $D$ that consists of an edge-disjoint combination of a directed forest $F$ and a directed cycle $C$ that connects the roots of the trees in $F$. The graph $D$ uniquely determines both $C$ and $F$, with $C$ being the sole simple cycle in the structure \cite{sahi2014harmonic}. Given a reach $R$ with $i \in R$, an  \textit{i-tree} then is a directed spanning tree of $R$ with root $i$.  An \textit{i-dangle} is a spanning dangle of $R$ where the cycle contains vertex $i$. 
The weight, $wt(\Gamma)$, of a subgraph $\Gamma$ of $G$ is defined as the product of the weights of all the edges contained in $\Gamma$.

\medskip 

Let $\Gamma$ be an i-dangle in $R$ with cycle $C$, and let $(j,i)$ and $(i,k)$ be the unique outgoing and incoming edges at vertex $i$ in $C$. Deleting one of these edges from $\Gamma$ gives rise to a $j$-tree and an $i$-tree, respectively. The dangle can be recovered uniquely from each of the two trees by reconnecting the respective edges. Thus, let $\mathcal{T}^R_i$ denote the set of $i$-trees for reach $R$. We obtain bijections from the set of $i$-dangles to the sets
$
\{ ((j,i), t) : t \in \mathcal{T}^R_j \}$ and $ \{ ((i,k), t) : t \in \mathcal{T}^R_i \}$ respectively.
Here $(j,i)$ and $(i,k)$ range over all outgoing and incoming edges at $i \in R$.

Thus, if $v_i$ is the weighted sum of all $i$-dangles, we have
$
\sum_{j\in R} \omega_R(j)  a(j,i) = v_i = \sum_{k \in R} \omega_R(i) a(i,k) 
$.
We can extend the left sum to a sum over all $j \in V$, since $\omega_R$ vanishes outside $R$. The sum on the right-hand side may similarly be extended to a sum over all $k \in V$: We first note that $\omega_R(i)$ is nonzero only if $i \in B(R)$. However, for such an $i$, all incoming edges must necessarily already originate in the cabal $B(R)$ of $R$ and hence in particular a part of the reach $R$. Thus  for $i \in B(R)$, $a(i,k)$ vanishes if $k \notin B(R)$ and we can extend the sum over $k$ to range over all of $V$.
Thus we have
\begin{align}
0 = \omega_R(i)\sum_{k \in R}   a(i,k)   -   \sum_{j\in R} \omega_R(j)  a(j,i)  = \omega_R(i) \din(i)  -   \sum_{j\in R} \omega_R(j)  a(j,i)  = \langle \omega_R, \Lin e_i \rangle_{\ell^2(G)}.
\end{align}
Since $i\in V$ was chosen arbitrarily, we indeed have $\langle \omega_R, \Lin \cdot\rangle_{\ell^2(G)} = 0$.

\medskip

To see the validity of the proposed normalization, note that $\omega_R \geq 0$, and as argued above $\omega_R(i) > 0$ precisely if $i \in B(R)$. By scaling if necessary, it is then clearly possible to normalize so that  $\sum_{j}\, \langle \cev \gamma^{-}_{R}, e_j \rangle =1.$ It remains to verify that with this normalization we also have $
\langle \cev\gamma^{-}_{R}, \vec\gamma^{-}_{S}\rangle_{\ell^2(G)} = \delta_{RS}$. Clearly $
\langle \cev\gamma^{-}_{R}, \vec\gamma^{-}_{S}\rangle_{\ell^2(G)} = 0 $ if $R \neq S$. Moreover if $R = S$, we have $
\langle \cev\gamma^{-}_{R}, \vec\gamma^{-}_{R}\rangle_{\ell^2(G)}  = \sum_{j \in B(R)}\, \langle \cev \gamma^{-}_{R}, e_j \rangle = \sum_{j \in V}\, \langle \cev \gamma^{-}_{R}, e_j \rangle =1$.
\end{proof}

\medskip

\noindent \textbf{Out-degree Laplacian}
Let us now consider the out-degree setting.    The first key observation we make, is that $\Lout$ on $G$ is the adjoint of
$\Lin$ on $G^\intercal$:
\begin{Prop}\label{in_out_transpose_prop}
	Let $\Lout_G$ be the out degree Laplacian on the graph $G$. Let $\Lin_{G^\intercal}$ be the in-degree Laplacian on the transposed graph $G^\intercal$. Then $\Lout_G = [\Lin_{G^\intercal}]^\star$; where '$^\star$' denotes the Hilbert space adjoint on $\ell^2(G)$.
\end{Prop}
\begin{proof}
	This is a straightforward calculation. For $f,g \in \ell^2(G)$, we find
	\begin{align}
		\langle f, \Lout_G g \rangle_{\ell^2(G)} &= \sum\limits_{i \in V} \overline{f}(i)
		\left( \sum_{v \in V}    a(v.i)    \right)g(i)  
		+
		\sum\limits_{i \in V}
		\overline{f}(i)
		\sum_{w \in V}    a(i.w)    g(w)  \\
		&= 
		\sum\limits_{i \in V} \overline{ \left( \sum_{v \in V}    a^\intercal(i,v)    \right) f(i)}
		g(i)  
		+
		\sum_{w \in V}  
		\overline{ \sum\limits_{i \in V} a^\intercal(w,i) f(i)}
		g(w)  \\
		&= 	\langle \Lin_{G^\intercal} f,  g \rangle_{\ell^2(G)},
	\end{align}
which proves the claim. Here we used $\ell^2(G) = \ell^2(V,m) = \ell^2(G^\intercal)$.
\end{proof}

In order to understand the left and right kernel structure of $\Lout$, we can thus make use of our knowledge about the structure of  left- and right kernels of in-degree Laplacians, gained above. We begin by applying the definition of  weight vectors to reaches $R^\intercal$ making up the transposed graph $G^\intercal$:

\begin{Def}\label{weight_vector_of_trranspose} Let $\mathcal{T}^{R^\intercal}_i$ be the set of all spanning trees of the reach $R^\intercal$ that are rooted at node $i\in R^\intercal$. Let  $\tau^\intercal_i$ be such a spanning tree beginning at node $i$. We define the weight-vector of the reach $R^\intercal$ as 
	$
	(w_R^\intercal)[i] = \sum\limits_{\tau^\intercal_i \in \mathcal{T}^{R^\intercal}_i} \prod\limits_{(x,y)\in\tau^\intercal_i} a^\intercal(x,y).
	$
\end{Def}

\begin{Ex}\label{weight_vector_example_writing_II}
	Let us revisit the graphs of Figure \ref{weight_vector_example}. We have depicted the corresponding transposed graphs in Fig. \ref{weight_vector_example_transposed} below. Graph $G_1^\intercal$ contains two reaches (as opposed to $G_1$ before transposing; which contained a single reach only). We denote the corresponding reaches by $R^\intercal = \{a,b\}$ and $S^\intercal = \{a,c,d\}$. The corresponding weight vectors are easily found to be given by $\omega_{R^\intercal}(a) = \alpha $, $\omega_{R^\intercal}(b) = \omega_{R^\intercal}(c) = \omega_{R^\intercal}(d) = 0$ and $\omega_{S^\intercal}(d) = \delta \cdot \gamma$, $\omega_S(a) = \omega_S(b) = \omega_S(c) = 0$. Finally graph $G_3^\intercal$ contains a single reach, with corresponding weight vector $\omega_{R^\intercal}(d) = \eta \cdot \alpha \cdot \delta + \delta \cdot \gamma \cdot \eta + \delta \cdot \rho \cdot \gamma$, $\omega_{R^\intercal}(a) = \omega_{R^\intercal}(b) = \omega_{R^\intercal}(c) = 0$.
\end{Ex}

\begin{figure}[H]
	\centering
	\begin{subfigure}[b]{0.25\textwidth}
		\centering
		\begin{tikzpicture}
			\node[draw, circle, minimum size=12pt, inner sep=0pt] (a) at (0,1) {$a$};
			\node[draw, circle, minimum size=12pt, inner sep=0pt] (b) at (2,2) {$b$};
			\node[draw, circle, minimum size=12pt, inner sep=0pt] (c) at (2,0) {$c$};
			\node[draw, circle, minimum size=12pt, inner sep=0pt] (d) at (4,1) {$d$};
			
			\draw[->] (b) -- node[above left] {$\alpha$} (a);
			\draw[->] (c) -- node[below left] {$\gamma$} (a);
			\draw[->] (d) -- node[below right] {$\delta$} (c);
		\end{tikzpicture}
		\caption{$G^\intercal_1$}
	\end{subfigure}
	\hspace{0.05\textwidth}
	\begin{subfigure}[b]{0.25\textwidth}
		\centering
		\begin{tikzpicture}
			\node[draw, circle, minimum size=12pt, inner sep=0pt] (a) at (0,1) {$a$};
			\node[draw, circle, minimum size=12pt, inner sep=0pt] (b) at (2,2) {$b$};
			\node[draw, circle, minimum size=12pt, inner sep=0pt] (c) at (2,0) {$c$};
			\node[draw, circle, minimum size=12pt, inner sep=0pt] (d) at (4,1) {$d$};
			
			\draw[->] (b) -- node[above left] {$\alpha$} (a);
			\draw[->] (d) -- node[below right] {$\delta$} (c);
		\end{tikzpicture}
		\caption{$G^\intercal_2$}
	\end{subfigure}
	\hspace{0.05\textwidth}
	\begin{subfigure}[b]{0.25\textwidth}
		\centering
		\begin{tikzpicture}
			\node[draw, circle, minimum size=12pt, inner sep=0pt] (a) at (0,1) {$a$};
			\node[draw, circle, minimum size=12pt, inner sep=0pt] (b) at (2,2) {$b$};
			\node[draw, circle, minimum size=12pt, inner sep=0pt] (c) at (2,0) {$c$};
			\node[draw, circle, minimum size=12pt, inner sep=0pt] (d) at (4,1) {$d$};
			
			\draw[->] (b) -- node[above left] {$\alpha$} (a);
			\draw[->] (c) -- node[below left] {$\gamma$} (a);
			\draw[->] (d) -- node[below right] {$\delta$} (c);
			\draw[->] (c) -- node[right] {$\rho$} (b);
			\draw[->] (d) -- node[above right] {$\eta$} (b);
		\end{tikzpicture}
		\caption{$G^\intercal_3$}
	\end{subfigure}
	\caption{Example graphs of Fig \ref{weight_vector_example}.  with reversed edge directions}
	\label{weight_vector_example_transposed}
\end{figure}
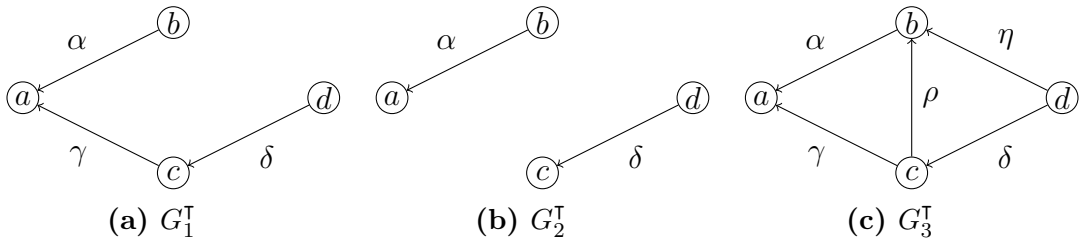

With Definition \ref{weight_vector_of_trranspose},  we find the following for the left-kernel of $\Lout$:

\begin{prop}\label{right_kernel_+}
	Let $G$ be a finite digraph, with  $ \mathfrak{R}^{\intercal} = \{R^\intercal\}$ the reaches of $G^\intercal$.  
	A basis of the \emph{right} kernel of the out-degree Laplacian $\Lout$ consists of elements $\{\vec\gamma^{+}_{R^\intercal}\}_{R^\intercal \in \mathfrak{R}^\intercal}$, with each element  parallel to the corresponding weight vector ($\vec\gamma^{+}_{R^\intercal} \parallel \omega_{R^\intercal}$), and
	  normalization	$\sum_{j \in V}\, \langle e_j, \vec\gamma^{+}_{R^\intercal}\rangle =1$, with 	$\langle e_j,	\vec	\gamma^{+}_{R^\intercal}\rangle >0$ iff $ j\in B(R^\intercal)$ and zero otherwise.
	\end{prop}

	\begin{proof}
By Proposition \ref{in_out_transpose_prop}, the right kernel of $\Lout_G$ is the same as the left-kernel of $\Lin_{G^\intercal}$. By Proposition \ref{prop:leftkernel_-}, a basis of the left kernel of $\Lin_{G^\intercal}$ is exactly given by the family $\{\vec\gamma^{+}_{R^\intercal}\}_{R^\intercal \in \mathfrak{R}^\intercal}$ specified above.
	\end{proof}
	
For the corresponding right-kernel, we have the following:

\begin{Prop}\label{prop:leftkernel_+}
	A basis of the \emph{left} kernel of an out-degree Laplacian $\Lout$ consists of elements $\{\langle \cev{\gamma}^{+}_{R^\intercal},\cdot \rangle_{\ell^2(G)}  \}_{R^\intercal \in \mathfrak{R}^\intercal}$ of the dual space $[\ell^2(G)]^* \cong \ell^2(G)$, so that:
	\begin{align}
		\left\{\begin{matrix}
			\cev\gamma^{+}_{R^\intercal}(j) =1 & \logif & j\in H(R^\intercal) &\textrm{(exclusive)}\\
			\cev\gamma^{+}_{R^\intercal}(j) \in (0,1) & \logif & j\in C(R^\intercal) &\textrm{(common)}\\
			\cev\gamma^{+}_{R^\intercal}(j) =0 & \logif & j\not\in R^\intercal &\textrm{(not in reach)}\\			
		\end{matrix} \right.
	\end{align}
		Furthermore, regarding  interplay with
	the
	right-kernel we have  $
	\langle \cev\gamma^{+}_{R^\intercal}, \vec\gamma^{+}_{S^\intercal}\rangle_{\ell^2(G)} = \delta_{R^\intercal S^\intercal}$. Finally for all $j \in V$ arbitrary, we have $\sum_{R^\intercal \in \mathfrak{R}^\intercal}\,	\cev\gamma^{+}_{R^\intercal}(j) = 1$.
\end{Prop}
\begin{proof}
By Proposition \ref{in_out_transpose_prop}, the left kernel of $\Lout_G$ is the same as the right-kernel of $\Lin_{G^\intercal}$. By Proposition \ref{right_kernel_Lin}, a basis of the right kernel of $\Lin_{G^\intercal}$ is exactly given by the family $\{\cev\gamma^{+}_{R^\intercal}\}_{R^\intercal \in \mathfrak{R}^\intercal}$ specified above.
\end{proof}

	\section{Convergence under increasing Connectivity I: Undirected Graphs}\label{undirected_convergence}
	
	We are now interested in the setting where the connectivity within certain clusters within $G$ tends to infinity. 
	We first consider only undirected graphs and begin by specifying what we mean by 'a collection of clusters' in this setting:
	\begin{Def}\label{clusters} Given a graph $G$, a collection of undirected clusters $\{\hat E_i\}_{i \in I}$ is a countable family  of edge sets ($\hat E_i \subseteq E$) with total set $\hat{E} = \cup_{i \in I} \hat{E}_i$ so that 
		\begin{enumerate}[label=(\roman*), topsep=0pt, itemsep=-1em]
			\item Each $\hat E_i$ has finite mass: $ \sum_i m(i) < \infty$\\
			\item $\{\hat E_i\}_{i \in I}$ is pairwise node-disjoint:  $(a,b) \in E_i \Rightarrow \forall v \in V$: $\nexists j \neq i$: $(a,v) \in \hat E_j$.\\
			\item Each $\hat E_i$ determines a connected component of $(V,\hat E)$: If there are edges $(a,b), (c,d) \in \hat E_i$, then there exists a path $a \leftrightsquigarrow d $ completely contained within $\hat E_i$.  
		\end{enumerate}
Associated to each cluster $\hat E_i$ is the corresponding node set $\hat{V}_i =  \{v| \exists j \in V: (v,a) \in \hat E_i  \}$ of the vertices connected by edges in the cluster  $E_i$.
	\end{Def}
In particular, node-disjointness clearly implies edge-disjointness ($\forall i \neq j: \hat E_i \cap \hat E_j = \emptyset$), and that we may always ensure (iii) by further subdivision of each $E_i$ if necessary.

	\subsection{Strong resolvent convergence}\label{undirected_strong_resolvent_convergence}
		We are now interested in the setting where the connectivity within the individual clusters $i \in I$ increases unboundedly.  More specifically, we consider a setting where we modify our initial graph $G = (G,m,E,a)$ into a family of graphs $G_{\beta} = (G,m,E,a_{\beta})$ indexed by $\beta \in [0,\infty)$ so that $G_0 = G$ and 
	\begin{align}\label{defining_kato_setting}
		\text{for $\tilde{\beta} >  \beta $}: 
		\quad
	\begin{cases}
				a_{\tilde{\beta}}(i,j) = a(i,j), & \text{if } (i,j) \in E\setminus \hat{E}  \\
	 a_{\tilde{\beta}}(i,j) > a_\beta(i,j), & \text{if }  (i,j) \in  \hat{E},
	\end{cases}
	\end{align}
with $a_{\beta}(i,j) \rightarrow \infty$ for any $(i,j) \in \hat{E}$, as $\beta \rightarrow \infty$.

\medskip

		With the Dirichlet form $	\mathcal{E}_{G_\beta}$ of (\ref{dirichlet_form}), we then have 	for $\tilde{\beta} > \beta$  that:
	\begin{align}\label{def_quadratic_form}
			\mathcal{E}_{G_\beta}(f,f) =  \frac12 \left(\sum\limits_{(i,j) \in E \setminus \hat{E}} a(i,j) |f(i) - f(j)|^2 +   \sum\limits_{(i,j) \in  \hat{E}} a_\beta(i,j) |f(i) - f(j)|^2 \right)\leq  	\mathcal{E}_{G_{\tilde{\beta}}}(f,f)
	\end{align}
This inequality is strict unless $f$ is constant on all  sets $\{V_i\}_{i \in I}$ of cluster nodes.
The prototypical example of this setting arises by modifying the edge-weight function of the original graph $G$ as $a \mapsto \hat{a}_\beta$, with 
\begin{align}\label{prototypical_example}
\hat{a}_{\beta}(x,y) =  a(x,y) + \beta\cdot \chi_{\hat{E}}((x,y)) a(x,y).
\end{align}
Here $\chi_{\hat{E}}$ is the characteristic function of the total-set $\hat{E}$ of cluster-edges.
	\medskip

We may decompose the  Laplacian
$L_\beta$ corresponding to $\mathcal{E}_{G_\beta}$ as 
\begin{align}\label{original_decomposition}
	L_\beta = L_{E \setminus \hat{E}} +  L_{\hat{E},\beta},
\end{align}
with the two constituent operators defined as

	\begin{align}\label{split_laplacians}
		[L_{\hat{E},\beta} f](x) &= \frac{1}{m(x)} \left(d_{\hat{E},\beta}(x)  f(x)   - \sum\limits_{y \in V}[\chi_{\hat{E}}((x,y)) a_\beta(x,y)] f(y) \right), \\
		[L_{E \setminus \hat{E}} f](x) &= \frac{1}{m(x)} \left(d_{E \setminus \hat{E}}(x)  f(x)   - \sum\limits_{y \in V}[\chi_{E \setminus \hat{E}}((x,y)) a(x,y)] f(y) \right). 
	\end{align}

Here $d_{\hat{E},\beta}(x) = \sum_{y \in V} \chi_{\hat{E}}((x,y)) a_\beta(x,y)$ represents the degree of a node $x \in V$ in the graph determined by the edges in $\hat{E}$, while $d_{E\setminus \hat{E}}(x) = \sum_{y \in V} \chi_{E\setminus\hat{E}}((x,y)) a(x,y)$ is the degree of the same node in the graph determined by the remaining edges $E \setminus \hat{E}$. Clearly $d_\beta(x) = d_{\hat{E},\beta}(x)  + d_{E\setminus \hat{E}}(x)$, with $d_\beta(x) = \sum_{y \in V} a_\beta(x,y)$ the degree of node $x$ in the original graph $G$. It is straightforward to verify, that
	the Laplacians of (\ref{split_laplacians}) are induced by the graph structures $G_{\hat{E}} = (V, m, \hat{E}, a_{\hat{E},\beta})$ and $G_{E \setminus \hat{E}} = (V, m, E\setminus \hat{E}, a_{E\setminus \hat{E}})$, with $a_{ \hat{E}, \beta}, a_{E\setminus \hat{E}}$ the restrictions of $a_{\beta}$ to the respective sets.

We can thus think of the Laplacian in (\ref{original_decomposition}) as being given by a fixed background Laplacian $L_{E \setminus \hat{E}}$ determined by the graph structure $G_{E \setminus \hat{E}}$,
subjected to a (large) perturbation $L_{\hat{E}, \beta}$ determined by the clusters in $G_{\hat{E}}$ on which edge weights diverge.

\medskip

Clearly the perturbing operator $L_{\hat{E},\beta}$ as well as the total Laplacian $L_\beta = L_{E \setminus \hat{E}} +  L_{\hat{E},\beta}$ diverge in norm ($\|L_\beta\|, \|L_{\hat{E},\beta}\| \rightarrow \infty$). To establish a well defined notion of convergence, we hence instead consider the corresponding family of \emph{resolvents} $\{(L_\beta - z \IG)^{-1}\}_{\beta}$, which remain uniformly bounded in $\beta$: $
	\|(L_\beta - z \IG)^{-1}\| \leq 1/\dist(z, \mathds{R}_+
$).

\medskip

Below, we will use a modified version of Kato's monotone convergence theorem 	\cite{kato1995perturbation,simon1978canonical}, to establish convergence of the family of resolvents $\{(L_\beta - z \IG)^{-1}\}$ towards the resolvent of a limit operator defined on a subspace of $\ell^2(G)$. We then interpret this limit operator in graph theoretic terms as an effective Laplacian defined on a 'coarse grained' graph $\underline{G}$. We begin by first introducing this coarse grained graph:

\begin{Def}\label{und_coarse_G_def} Given a graph $G = (V, m, E, a)$ and a countable collection $\{E_i\}_{i \in I}$ of clustered edges,
	we define a reduced graph $\underline{G} = \{\underline{V}, \underline{m}, \underline{E},\underline{a}\}$ as follows:
	Define an equivalence relation "$\sim_{\hat{E}}$" on the node-set $V$ by setting $a \sim_{\hat{E}} b $ iff
	 $a$ and $b$ are contained in the same connected component of the graph $(V, \hat{E})$.
	\begin{enumerate}[label=(\roman*), topsep=0pt, itemsep=-.2em]
		\item 	The node set $\underline{V}$, is taken to be the set of connected components $\mathfrak{C}_{\hat{E}}$ of $(V, \hat{E})$:
		\begin{align}
			\underline{V} : =  \mathfrak{C}_{\hat{E}} = V/ \sim_{\hat{E}}
		\end{align}
		\item The weight of a node $C \in \underline{V}$ is given as the sum of the weights in the corresponding connected component $C \subseteq V$ of the original graph $(V, \hat{E})$: 
		\begin{align}\label{mass_def_undirected}
			\underline{m}(C) := \sum\limits_{v \in C} m(v).
		\end{align}
		\item The edge-weight function $\underline{a}:\underline{V} \times \underline{V} \rightarrow [0, \infty)$ arises similarly via aggregation:
		\begin{align}\label{edge_def_undir}
			\underline{a}(C, D) := \sum\limits_{v \in C, w \in D} a(v,w)
		\end{align}
		\item The set of edges $\underline{E}$ is then given as the support of $\underline{a}$ on $\underline{V} \times \underline{V}$:
		\begin{equation}
			\underline{E} = \{(C,D)\in\underline{V}\times\underline{V}: \underline{a}(C,D) >0 \}
		\end{equation}
	\end{enumerate}
\end{Def}
Here we have denoted by $(V,\hat{E})$ the graph with node set $V$, edge set $\hat{E}$ and all non-zero edge weights and node masses set to one. Note that a connected component of $(V,\hat{E})$ may also  consist only of a single node not connected to any other node 
via an edge in $\hat{E}$.

\medskip
 We can  canonically  associate a Laplacian to the graph $\underline{G}$ via Definition \ref{laplacian_definition}. Since this operator is defined on the Hilbert space $\ell^2(\underline{G})$  as opposed to $\ell^2(G)$, we need \emph{projection}- and  \emph{interpolation} operators translating between these two spaces:

 \begin{Def}\label{proj_interp_undirected}
 	At the level of Hilbert spaces, we define the bounded projection- $J_\downarrow: \ell^2(G) \rightarrow \ell^2(\underline{G})$ and interpolation- $J_\uparrow: \ell^2(\underline{G}) \rightarrow \ell^2(G)$ operators as 
 	\begin{align}\label{undirected_Js}
 		[J_\downarrow f](C) = \langle \chi_{C}, f \rangle_{\ell^2(V)}/\underline{m}(C),\quad \text{and} \quad J_\uparrow \underline{f} = \sum_{v \in C} \underline{f}(C) \chi_{C}.
 	\end{align}
 With this Definition, $J_\uparrow J_{\downarrow}= P$, with $P$ the orthogonal projection onto the subspace of $\ell^2(G)$ given by elements $f$ which are constant on each $V_i$ ($i \in I$). This space is precisely the kernel of the Laplacian $L_{\hat{E},\beta}$ (for any $\beta$).
 \end{Def}

 In principle, there is one scalar degree of fredom per cluster $C$ in the definition of $J_{\uparrow\downarrow}$: For any collection $\{b_C\}_{C \in \mathfrak{C}}$ with    $b_C \neq 0$, making the simultaneous change $[J_\downarrow f](C) \mapsto b_c \cdot \langle \chi_{C}, f \rangle_{\ell^2(V)}/\underline{m}(C)$ and  $J_\uparrow \underline{f} \mapsto \sum_{v \in C} \underline{f}(C)/b_C$   leads to the same projection $P = J_\uparrow J_\downarrow $.
For the choice $b_C \equiv 1$ of Definition \ref{proj_interp_undirected}, it is not hard to see that $J_\uparrow: \ell^2(\underline{G}) \rightarrow \ell^2(G) $ is isometric and that we have $\|J_\downarrow\| \leq 1$.  However, even more can be said about the naturality of $J_{\uparrow \downarrow}$

\begin{Rem}
	 Denote by $\ell^1(G)$ the Banach space of functions on $G$ that are summable  with respect to $m$ ($\sum_{v \in V} |f(v)| m(v) < \infty$). Any such function $f \in \ell^2(G) \cap \ell^1(G)$ that only takes on non-negative values  and satisfies $\|f\|_{\ell^1(G)} = 1$ may then be interpreted as a probability distribution. It is well known that the heat kernel preserves the property of 'being a probbability distribution': If $f \geq 0$, then if $\|f\|_{\ell^1(G)} = 1$ also  $\|e^{-tL}f\|_{\ell^1(G)} = 1$ for any $t \geq 0$. Taking the limit $t \rightarrow \infty$ and using that $e^{-tL_{\hat{E},\beta}} \rightarrow P$ in norm, we also find $ 1 = \|f\|_{\ell^1(G)} = \|Pf\|_{\ell^1(G)} = \|J_\uparrow J_\downarrow f\|_{\ell^1(G)} $. The maps $J_{\uparrow \downarrow} $ are now chosen precisely so that  they map probability distributions to probability distributions: If $f \geq 0$ then $[J_\downarrow f] \geq 0$ and similarly if $\ell^2(\underline{G}) \cap \ell^1(\underline{G}) \ni \underline{f} \geq 0$ then $[J_\uparrow f] \geq 0$.  Furthermore we have $\|J_\downarrow f\|_{\ell^1(\underline{G})} = \|f\|_{\ell^1(G)}$ and 
	 $ \|J_\uparrow \underline{f}\|_{\ell^1(G)}  = \| \underline{f}\|_{\ell^1(\underline{G})} $.
\end{Rem}

\medskip
With $\underline{G}$ as in Definition \ref{und_coarse_G_def} we find:
\begin{Thm}\label{undir_appr_thm_strongly}
	Let $\underline{L}$ be the Laplacian on $\underline{G}$. We have the \emph{strong} convergence
	\begin{equation}\label{strong_undirected_resolvent_closeness}
		(L_\beta - z I_{\ell^2(G)})^{-1} \rightarrow  J^\uparrow  (\underline{L} - z I_{\ell^2(\underline{G})})^{-1} J^\downarrow.
	\end{equation}
\end{Thm}

\begin{proof}
	The proof proceeds in two steps: We first establish generalized strong resolvent convergence of $L_\beta$ towards a limit operator $T$ defined on a subspace of $\ell^2(G)$. Then we use projection and interpolation operators as introduced in Definition \ref{proj_interp_undirected} to identify $T$ with a coarse grained graph Laplacian $\underline{L}$ on $\ell^2(\underline{G})$ as $T = J^\uparrow \underline{L} J^\downarrow$.
	
	\medskip

	To establish convergence towards a limit operator $T$, we first define the limit quadratic form 
		$
		q(f,f) = \lim_{\beta \rightarrow \infty} \mathcal{E}_{G_\beta}(f,f)
		$
whenever the finite limit exists. By \cite[VIII Theorem 3.13a]{kato1995perturbation}, $q$ is a closed symmetric form bounded from below. 
Clearly
\begin{align}\label{well_def_cond}
	 \lim\limits_{\beta \rightarrow \infty} \mathcal{E}_{G_\beta}(f,f) < \infty \Leftrightarrow \text{$f$ is constant on each $V_i$, $i \in I$}.
\end{align}
Let  now $P$ be the orthogonal projection onto $\{f \in \ell^2(G): \text{$f$ is constant on each $V_i$}\}$; i.e. the closed space of elements satisfying (\ref{well_def_cond}).
Then $q$ is a bounded quadratic form with domain all of $P\ell^2(G)$. Denote by $T$ the operator associated to $q$ via $q(g,g) = \langle g,Tg\rangle_{P\ell^2(G)}$. By \cite[Theorem 4.1]{simon1978canonical} we hence have the strong convergence $(L_\beta - z \IG)^{-1} \rightarrow P(T - z I_{P \ell^2(G)})^{-1}P$.

\medskip
Using positivity of $q$, polarization, the defining relation $
	q(f,f) =: \langle f, T f\rangle_{P \ell^2(G)}$
and the fact that for any $f \in P \ell^2(G)$ we have $q(f,f) = \langle f, L_\beta f\rangle_{\ell^2(G)} = \langle f, L_{E \setminus \hat{E}} f\rangle_{\ell^2(G)}$, it is straightforward to verify that $T = PL_{E \setminus \hat{E}}P =  PL_{\beta}P. $
\medskip Thus it remains to verify that
\begin{align}\label{intermediate_objective}
P(PLP - z I_{P \ell^2(G)})^{-1}P =  J_\uparrow  (\underline{L} - z I_{\ell^2(\underline{G})})^{-1} J_\downarrow.
\end{align}

Arguing from right to left in (\ref{intermediate_objective}),  we note that since \( J_\uparrow \) is a right-inverse for \( J_\downarrow \) and in the opposite direction \( J_\downarrow \) inverts \( J_\uparrow \) on its range, we have 
\begin{align}
	J_\uparrow  (J_\downarrow  L_{E \setminus \hat{E}}  J_\uparrow - z    I_{\ell^2(\underline{V})})^{-1} J_\downarrow &=
	J_\uparrow  ((J_\downarrow J_\uparrow) J_\downarrow  L_{E \setminus \hat{E}}  J_\uparrow  (J_\downarrow J_\uparrow)- z    (J_\downarrow J_\uparrow))^{-1} J_\downarrow
	\\
	=& 
	J_\uparrow  (J_\downarrow [J_\uparrow J_\downarrow  L_{E \setminus \hat{E}}  J_\uparrow  J_\downarrow - z  I_{\ell^2(V)}  ]J_\uparrow)^{-1} J_\downarrow
	\\
	=& J_\uparrow J_\downarrow  (J_\uparrow J_\downarrow  L_{E \setminus \hat{E}}  J_\uparrow J_\downarrow - z    I_{\ell^2(V)})^{-1}J_\uparrow J_\downarrow.
\end{align}
However, the product \( J_\uparrow J_\downarrow \) is simply the  projection \( P \) above, as is easily verified. 

\medskip

It remains to verify that the Laplacian $\underline{L}$ associated to $\underline{G}$ is exactly given as $\underline{L} = J_\downarrow L_{E\setminus \hat{E}} J_\uparrow$. But this is a straightforward calculation.
\end{proof}

From the proof of Theorem \ref{proj_interp_undirected}, we can now also understand the demand that the mass $m(V_i) = \sum_{v \in V_i} m(v)$ of each node set associated to the cluster $E_i$ be finite:
Pick an element $f \in \ell^2(G)$. From (\ref{def_quadratic_form}), we see that  $\lim_{\beta \rightarrow \infty}	\mathcal{E}_{G_\beta}(f,f) < \infty$ is possible  only if $f$ is constant on $V_i$. However, if $m(V_i)  = \infty$ any element $f$ constant on $V_i$ is no longer in $\ell^2(G)$. Assuming finite mass of clusters precisely avoids this degenerate case.

\begin{Rem}\label{naturality_remark_undirected}
	For future reference, we here note, that we may also express the definition of aggregated node masses and edge weights in a way that makes their relation to the (kernel of the) perturbing Laplacian $L_{\hat{E},\beta}$  manifest: The right kernel of $L_{\hat{E},\beta}$ is spanned by the vectors  $\{\vec{\gamma}_C\}_{C \in \mathfrak{C}}$ which are the indicator functions of the various clusters $C \in \mathfrak{C}$  (i.e. $\vec{\gamma}_C = {\chi_C}$). Since $L_{\hat{E},\beta}$ is self-adjoint on $\ell^2(G)$, its left kernel is canonically isomorphic to the right kernel via the Riesz representation theorem. We choose a basis $\{\cev\gamma_C\}_{C \in \mathfrak{C}}$, so that the basis elements $\langle \cev{\gamma}_C, \cdot \rangle_{\ell^2(G)}$ of the left kernel satisfy  $\langle \cev{\gamma}_C, \vec{\gamma}_B \rangle_{\ell^2(G)} = \delta_{CB}$. The orthogonal projection $P$  onto the kernel of $L_{\hat{E}, \beta}$ may then be written as
	$
	P = \sum_{C \in \mathfrak{C}} \vec{\gamma}_C \langle \cev{\gamma}_C, \cdot \rangle_{\ell^2(G)}
	$.\footnote{ This is exactly the spectral projection onto $0 \in \sigma(L_{\hat{E},\beta})$.}
	With this, we have
	\begin{align}
		\underline{m}(C) = \langle \chi_C, P \chi_C \rangle_{\ell^2(G)}, \quad 
		\underline{a}(C, B) := \underline{m}(C)\cdot\sum\limits_{c \in C, b \in B} 
		[
		\cev{\gamma}_{C}(c)
		a(c,b)
		\vec{\gamma}_B(b)
		].
	\end{align} 
\end{Rem}

	\medskip
	Let us  consider two fundamental examples to which Theorem \ref{undir_appr_thm_strongly} applies:
	
	\begin{Ex}\label{one_edge_div}	Let us first consider the setting of a single edge\\
\noindent		\begin{minipage}[b]{0.7\textwidth}
 tending to infinity. More  
 precisely (c.f. also Fig. \ref{fig:graphs}), consider a graph on three nodes ($V = \{a,b,c\}$), with edge weights  $a(a,b) = a(a,c) = 1$ and $a(b,c) = \beta \rightarrow \infty$, and thus are in the setting of (\ref{prototypical_example}). 
With node mass matrix $M =  \text{diag}(m(a), m(b), m(c))$, we have
\begin{align}\label{und_one_edge_div_laplacian}
	L_\beta =
	M^{-1}
	\cdot 	
	\left(
	\begin{pmatrix}
		2 & -1 & -1 \\
		-1 &1 & 0 \\
		-1 & 0 & 1
	\end{pmatrix}
	+ \beta
	\begin{pmatrix}
		0 & 0 & 0 \\
		0&1 & -1 \\
		0 & -1 & 1
	\end{pmatrix}
	\right).
\end{align}

The reduced node set is given as $\underline{V} = \{a, \{b,c\}\}$ and we have $\underline{m}(a) = m(a)$ and  $\underline{m}( \{b,c\}) = m(b) + m(c)$. We find 
\noindent
		\end{minipage}
		\begin{minipage}[b]{0.4\textwidth}
			\centering
			\begin{minipage}[b]{0.6\textwidth}
\raggedleft
				\begin{tikzpicture}
					\node[draw, circle, minimum size=12pt, inner sep=0pt] (a) at (1,1.73) {$a$};
					\node[draw, circle, minimum size=12pt, inner sep=0pt] (b) at (0,0) {b};
					\node[draw, circle, minimum size=12pt, inner sep=0pt] (c) at (2,0) {c};
					\draw (a) -- node[left] {$1$} (b);
					\draw[red] (b) -- node[black, below] {$\beta \gg 1$} (c);
					\draw (c) -- node[right] {$1$} (a);
				\end{tikzpicture}
			
				\par\vspace{0.5ex}
				\centering{	(a) $G$}
			\end{minipage}
			\hfill
			\begin{minipage}[b]{0.35\textwidth}
				\raggedright
				\begin{tikzpicture}
					\node[draw, circle, minimum size=12pt, inner sep=0pt] (a) at (1,1.73) {$a$};
					\node[draw, circle, minimum size=12pt, inner sep=0pt] (bc) at (1,0) {$b \cup c$};
					\draw (a) -- node[right] {$2$}  (bc);
				\end{tikzpicture}
				\par\vspace{0.5ex}
				(b) $\underline{G}$
			\end{minipage}
			\captionof{figure}{Original graph $G$ and reduced graph $\underline{G}$.}
			\label{fig:graphs}
		\end{minipage}

\noindent  
$J_\downarrow = (1,0)^\intercal\cdot (1,0,0) + (0,1)^\intercal \cdot (0,m(b), m(c))/[m(b) + m(c)]$ and
$J_\uparrow = (1,0,0)^\intercal \cdot (1,0) + (0,1,1)^\intercal \cdot (0,1) $. Finally, with $\underline{M}  = \text{diag}(m(a) , [m(b) + m(c)])$ we find the reduced Laplacian $\underline{L}$ acting  on $\ell^2(G)$ to be given as
\begin{align}
L = \underline{M}^{-1} \cdot	\begin{pmatrix}
		2 & -1 \\
		-2 & 2
	\end{pmatrix}.
\end{align}
	\end{Ex}

	\begin{Ex}\label{zero_is_accumulation_point} Next let us consider an infinite graph, with node set $V = \{\nu_i\}_{i = 1}^\infty$, node weights set to unity, and edge weights given as $a(\nu_{2n-1},\nu_{2n}) =1$ and $a(\nu_{2n},\nu_{2n+1}) =\beta/(2n)$ (c.f. Fig. \ref{fig:pathgraph}). We are then interested in the setting $\beta \rightarrow \infty$.
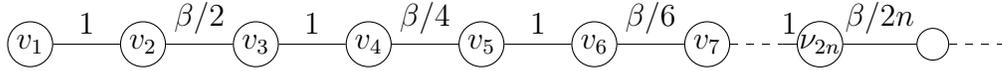
\begin{figure}[H]
\begin{tikzpicture}
	\foreach \i [evaluate=\i as \j using int(\i+1)] in {0,...,6} {
		\node[draw, circle, minimum size=17pt, inner sep=0pt] (v\j) at (\i*1.5,0) {$v_{\j}$};
	}
	
	\draw (v1) -- node[above] {$1$} (v2);
	\draw (v2) -- node[above] {$\beta/2$} (v3);
	\draw (v3) -- node[above] {$1$} (v4);
	\draw (v4) -- node[above] {$\beta/4$} (v5);
	\draw (v5) -- node[above] {$1$} (v6);
	\draw (v6) -- node[above] {$\beta/6$} (v7);
	
	\draw[dashed] (v7) -- ++(1.0,0);
	
	\node[draw, circle, minimum size=12pt, inner sep=0pt] (x1) at (10.5,0) {$\nu_{2n}$};
	\draw (v7) ++(1.0,0) -- node[above] {$1$} (x1);
	
	\node[draw, circle, minimum size=12pt, inner sep=0pt] (x2) at (12,0) {};
	\draw (x1) -- node[above] {$\beta/2n$} (x2);
	
	\draw[dashed] (x2) -- ++(1.0,0);
\end{tikzpicture}
	\captionof{figure}{Path graph $G$}
\label{fig:pathgraph}
\end{figure}
 \noindent It is not hard to see, that the limit graph is given by the node set $\underline{V} = \{\mu_i\}_{i = 1}^\infty$ with $\mu_i = \{\nu_{2i },\nu_{2i +1 }\}$ for $i > 1$ and $\mu_1 = \{\nu_1\}$. Node weights are set to $\underline{m}(\mu_i) = 2$ if $i > 1$  and $\underline{\mu_1} = 1$. Edge weights are given as $\underline{a}(\mu_i, 
 \mu_{i + 1}) = 1$ (c.f. also Fig. \ref{fig:pathgraph_reduced} below). 
  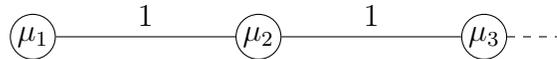
\begin{figure}[H]
\begin{tikzpicture}
	\foreach \i [evaluate=\i as \j using int(\i+1)] in {0,...,2} {
		\node[draw, circle, minimum size=17pt, inner sep=0pt] (v\j) at (\i*3.0,0) {$\mu_{\j}$};
	}
	
	\draw (v1) -- node[above] {$1$} (v2);
	\draw (v2) -- node[above] {$1$} (v3);
\draw[dashed] (v3) -- ++(1.0,0);
\end{tikzpicture}
	\captionof{figure}{Reduced path graph $\underline{G}$}
\label{fig:pathgraph_reduced}
\end{figure}

	\end{Ex}

We also note that we can now rigorously link our convergence results to our intuition about heat flow from Section \ref{intro}:

\begin{Cor}\label{norm_dynamical_convergence_strong}
	For any $t > 0$, we have $e^{-tL_\beta} \rightarrow J_\uparrow e^{-t\underline{L}} J_\downarrow $ strongly as $\beta \rightarrow \infty$.
\end{Cor}
\begin{proof}
Since $L_\beta$ is bounded and non-negative, it generates a bounded holomorphic semigroup.
The claim is then a direct consequence of established results in approximation theory of degenerate semigroups, which guarantee that for uniformly bounded holomorphic semigroups, generalized strong resolvent convergence of generators translates to strong convergence of the associated semigroup \cite[Theorem 5.2]{arendt2001approximation}. 
\end{proof}
	
	\subsection{Norm Resolvent Convergence}\label{undirected_norm_resolvent_convergence}
We are now interested in determining under which conditions exactly, we can turn the \emph{strong} convergence of (\ref{strong_undirected_resolvent_closeness}) into norm resolvent convergence.
To this end, we first recall the definition of a spectral gap:
\begin{Def}The spectral gap of a non-negative operator $L$ is
$
	\gamma(L): = \inf\limits_{\lambda \neq 0} \sigma(L).
$
\end{Def}
Below, we will give a norm estimate on the difference
of $	(L_\beta - z I_{\ell^2(G)})^{-1}$  and $ J^\uparrow  (\underline{L} - z I_{\ell^2(\underline{G})})^{-1} J^\downarrow$ in terms of the spectral gap of $L_{\hat{E},\beta}$.
We  first establish a Lemma:

\begin{Lem}\label{conv_to_projection}
	 With $P$ the orthogonal spectral projection onto   $\ker(L_{\hat{E}})$, we have
	\begin{align}\label{convergence_to_projection}
		\left\|( L_{\hat{E},\beta} - z)^{-1} - P/(-z) \right\| =
1/|\gamma(L_{\hat{E},\beta}) - z|.
	\end{align}
	Moreover, with $\hat{E} = \cup_i \hat{E}_i$ as in Definition \ref{clusters}, we have $\gamma(L_{\hat{E}, \beta}) = \inf_{i \in I} \gamma(L_{\hat{E}_i, \beta})$.
\end{Lem}
\begin{proof}
This is a straightforward application of the spectral theorem. The claim  $\gamma(L_{\hat{E}, \beta}) = \inf_{i \in I} \gamma(L_{\hat{E}_i, \beta})$ follows, since $L_{\hat{E}, \beta} = \sum_{i \in I } L_{\hat{E}_i, \beta}$ and elements of the family $\{L_{\hat{E}_i, \beta}\}_{i \in I}$ are pairwise commuting.
\end{proof}
	The norm resolvent estimate we then prove is the following :

\begin{Thm}\label{undir_appr_thm}
Let $\underline{L}$ be the Laplacian associated to $\underline{G}$. For any $z \in \mathds{C} \setminus [0, \infty)$:
	\begin{equation}\label{resolvent_closeness_app}
		\left\|(L_\beta - z I_{\ell^2(V)})^{-1} -  J^\uparrow  (\underline{L} - z I_{\ell^2(\underline{V})})^{-1} J^\downarrow\right\| 
\lesssim	\left\|( L_{\hat{E},\beta} - z)^{-1} - P/(-z) \right\|  \lesssim
			 1/\gamma(L_{\hat{E}, \beta}).
	\end{equation}
\end{Thm}
\begin{proof}
	By a straightforward calculation, we have that the Laplacian $\underline{L}$  associated via Definition \ref{laplacian_definition} to the graph $\underline{G}$ of Definition \ref{und_coarse_G_def}  is exactly given as $	\underline{L} = J^\downarrow  L_{E \setminus \hat{E}}  J^\uparrow$,
	with projection $J^\downarrow$ and interpolation $J^\uparrow$ as introduced in Definition \ref{proj_interp_undirected}. We are thus interested in establishing 
	\begin{equation}
		\left\|(L_\beta - z    I_{\ell^2(V)})^{-1} -  J^\uparrow  (J^\downarrow  L_{E \setminus \hat{E}}  J^\uparrow - z    I_{\ell^2(\underline{V})})^{-1} J^\downarrow\right\|\leq 
		C_{L_{E\setminus \hat{E}}}	\left\|( L_{\hat{E},\beta} - z)^{-1} - P/(-z) \right\|
	\end{equation}
	From the proof of Theorem \ref{undir_appr_thm_strongly}, we know that 
	\begin{align}
		J_\uparrow  (J_\downarrow  L_{E \setminus \hat{E}}  J_\uparrow - z    I_{\ell^2(\underline{G})})^{-1} J_\downarrow 
		=
P (P  L_{E \setminus \hat{E}}  P- z    I_{\ell^2(G)})^{-1}JP.
	\end{align}
Thus it suffices to prove
		\begin{equation}
		\left\|(L_\beta - z I_{\ell^2(V)})^{-1} -  P  (PL_{E \setminus \hat{E}} P - z I_{\ell^2(V)})^{-1} P\right\|\leq  
			C_{L_{E\setminus \hat{E}}}	\left\|( L_{\hat{E},\beta} - z)^{-1} - P/(-z) \right\| 
	\end{equation}
	Let us write 
	\begin{align}
		L_{E \setminus \hat{E}} P -z I_{\ell^2(G)} = 
			\begin{pmatrix}
				PL_{E \setminus \hat{E}} P -zP & 0   \\
			QL_{E \setminus \hat{E}} P    & -z Q
		\end{pmatrix}.
	\end{align}
With $ Q = I_{\ell^2(V)} - P$	we then find 
	\begin{align}\label{schur_inverse_resolvent}
			(L_{E \setminus \hat{E}} P -z I_{\ell^2(G)})^{-1}
				=
				\begin{pmatrix}
					 (PLP - zP)^{-1} P & 0 \\
					\frac{1}{z} Q L P  (PLP - zP)^{-1} P & -\frac{1}{z} Q
				\end{pmatrix}
	\end{align}
	by well known formulas for lower triangular block-inversion. Thus $	(L_{E \setminus \hat{E}} P -z I_{\ell^2(G)})$ is invertible on  $\ell^2(G)$ whenever $(PLP - zP)$ is invertible on  $\ell^2(\underline{G})$ This is true   in particular whenever $z \in \mathds{C}\setminus[0,\infty)$. Moreover, since $P$ is an orthogonal projection we may choose a $z \in \mathds{C}\setminus [0,\infty)$ so that  $|z| > \|L_{E \setminus \hat{E}}\| \geq  \|L_{E \setminus \hat{E}}P\| \geq \|PL_{E \setminus \hat{E}}P\|$. Then x $PL_{E \setminus \hat{E}} P - z$ and $L_{E \setminus \hat{E}} P - z$  are invertible on $\ell^2(G)$ by a Neumann series argument:
\begin{align}
	&P(L_{E \setminus \hat{E}} P -z I_{\ell^2(G)})^{-1} = P \frac{1}{z} \sum\limits_{k = 0}^{\infty} \left[ \frac{L_{E \setminus \hat{E}} P}{z}\right]^k \\
	= &P \frac{1}{z} \sum\limits_{k = 0}^{\infty} \left[ \frac{P L_{E \setminus \hat{E}} P}{z}\right]^k = P(P L_{E \setminus \hat{E}} P -z I_{\ell^2(G)})^{-1}.
\end{align}
By the identity theorem for Banach-valued holomorphic functions (e.g.  \cite[Proposition A.2]{Arendt2001}), it then follows that $P(L_{E \setminus \hat{E}} P -z I_{\ell^2(G)})^{-1} = P(P L_{E \setminus \hat{E}} P -z I_{\ell^2(G)})^{-1}$ for all $z \in \mathds{C}\setminus[0,\infty)$. Similarly, we also find $P(L_{E \setminus \hat{E}} P -z)^{-1} Q = 0$.

\medskip

Our goal is thus now to approximate $P(L_{E \setminus \hat{E}} P -z)^{-1} $ with $( L_{E \setminus \hat{E}} + L_{\hat{E},\beta} -z)^{-1} = [-z(L_{\hat{E},\beta} -z)^{-1}](L_{E \setminus \hat{E}}[-z(L_{\hat{E},\beta} -z)^{-1}] -z)^{-1} $. To this end, we first note for $\gamma(L_{E \setminus \hat{E}})  \gg 1$ sufficiently large, we by Lemma \ref{conv_to_projection}:
\begin{align}
	&\| P (L_{E \setminus \hat{E}} P -z)^{-1} - (-z)(L_{\hat{E},\beta} - z)^{-1} (L_{E \setminus \hat{E}} P -z)^{-1}   \|\\
	\leq& \| (L_{E \setminus \hat{E}} P - z)^{-1}\| 
	\cdot
	\| P - (-z)(L_{\hat{E},\beta} - z)^{-1}  \| \\
	=&  \| (L_{E \setminus \hat{E}} P - z)^{-1}\| 
	\cdot |z|/|\gamma(L_{E \setminus \hat{E}}) - z|
\end{align}
Next we note that for $\gamma(L_{\hat{E},\beta}) \gg 1$ sufficiently large, we have that $-z L_{E \setminus \hat{E}} (L_{\hat{E},\beta} -z)^{-1} - z = L_{E \setminus \hat{E}} P -z + L_{E \setminus \hat{E}} [-z (L_{\hat{E},\beta} -z)^{-1}- P]$ is boundedly invertible.
Indeed, above we already esgtablished bounded invertibility of $L_{E \setminus \hat{E}} P -z $ on $\ell^2(G)$. Next we note that we have 
$
	\|L_{E \setminus \hat{E}} [-z (L_{\hat{E},\beta} -z)^{-1}- P]\| \leq \|L_{E \setminus \hat{E}} \| \cdot |z|/|\gamma(L_{\hat{E},\beta})-z|
$,
so that the above expression can be made arbitrarily small by choosing $\gamma(L_{\hat{E},\beta}) \gg 1$ large enough. 

 A Neumann argument then establishes invertibility of $-z L_{E \setminus \hat{E}} (L_{\hat{E},\beta} -z)^{-1} - z$:
\begin{align}
	&(-z L_{E \setminus \hat{E}} (L_{\hat{E},\beta} -z)^{-1} - z)^{-1}\\
	 = &(L_{E \setminus \hat{E}} P -z )^{-1} \sum\limits_{k = 0}^{\infty} [-(L_{E \setminus \hat{E}} P -z )^{-1} L_{E \setminus \hat{E}}  [-z (L_{\hat{E},\beta} -z)^{-1}- P]]^k
\end{align}
Hence for $\gamma(L_{\hat{E},\beta}) \gg 1$ sufficiently large, $-z L_{E \setminus \hat{E}} (L_{\hat{E},\beta} -z)^{-1} (-z L_{E \setminus \hat{E}} (L_{\hat{E},\beta} -z)^{-1} - z)^{-1}$ is well defined and
$
[L_{E \setminus \hat{E}} + L_{\hat{E},\beta} -z]\cdot[-z(L_{\hat{E},\beta} -z)^{-1} (-z L_{E \setminus \hat{E}} (L_{\hat{E},\beta} -z)^{-1} - z)^{-1}] = Id_{\ell^2(G)}$.
Furthermore we may estimate:
\begin{align}
	&\left\| (L_{E \setminus \hat{E}} + L_{\hat{E},\beta} - z I)^{-1} - P ( L_{E \setminus \hat{E}} P - z I)^{-1} P \right\|\\
	=& \| [-z(L_{\hat{E},\beta} -z)^{-1} (-z L_{E \setminus \hat{E}} (L_{\hat{E},\beta} -z)^{-1} - z)^{-1}] -  P ( L_{E \setminus \hat{E}} P - z I)^{-1}    \| \\
	\leq &     
	\| [-z(L_{\hat{E},\beta} -z)^{-1} (-z L_{E \setminus \hat{E}} (L_{\hat{E},\beta} -z)^{-1} - z)^{-1}] -  [-z(L_{\hat{E},\beta} -z)^{-1}( L_{E \setminus \hat{E}} P - z I)^{-1} ]   \| \\
	+& \| [-z(L_{\hat{E},\beta} -z)^{-1}( L_{E \setminus \hat{E}} P - z I)^{-1}  -  P ( L_{E \setminus \hat{E}} P - z I)^{-1}    \|
\end{align}

For the first term, we find
\begin{align}
& 	\| [-z(L_{\hat{E},\beta} -z)^{-1} (-z L_{E \setminus \hat{E}} (L_{\hat{E},\beta} -z)^{-1} - z)^{-1}] -  [-z(L_{\hat{E},\beta} -z)^{-1}( L_{E \setminus \hat{E}} P - z I)^{-1} ]   \|\\
\leq & |z| \cdot \| (L_{\hat{E},\beta} -z)^{-1} \| \cdot \|
(-z L_{E \setminus \hat{E}} (L_{\hat{E},\beta} -z)^{-1} - z)^{-1} - ( L_{E \setminus \hat{E}} P - z I)^{-1} 
  \|\\
 \leq & |z| /\dist(z, \mathds{R}_+) \cdot
 \left\|(L_{E \setminus \hat{E}} P -z )^{-1} \sum\limits_{k = 1}^{\infty} [-(L_{E \setminus \hat{E}} P -z )^{-1} L_{E \setminus \hat{E}}  [-z (L_{\hat{E},\beta} -z)^{-1}- P]]^k\right\|\\
 \leq & 2|z| /\dist(z, \mathds{R}_+) \cdot \|   (L_{E \setminus \hat{E}} P -z )^{-1}     \|^2
 \cdot
 \|L_{E \setminus \hat{E}}\| \cdot \|-z (L_{\hat{E},\beta} -z)^{-1}- P\|.
\end{align}
From (\ref{schur_inverse_resolvent}) we infer $ \|   (L_{E \setminus \hat{E}} P -z )^{-1}     \| \leq [1/\dist(z, \mathds{R}_+) + 1/|z| + \|Q L P \|/(\dist(z, \mathds{R}_+) \cdot |z|) ]$ and thus together with Lemma \ref{conv_to_projection} we find
\begin{align}
	& 	\| [-z(L_{\hat{E},\beta} -z)^{-1} (-z L_{E \setminus \hat{E}} (L_{\hat{E},\beta} -z)^{-1} - z)^{-1}] -  [-z(L_{\hat{E},\beta} -z)^{-1}( L_{E \setminus \hat{E}} P - z I)^{-1} ]   \|\\
	\leq & \frac{2 |z|^2}{\dist(z, \mathds{R}_+)}  \left(\frac{1}{\dist(z, \mathds{R}_+)} + \frac{1}{|z|} + \frac{\|QLP\|}{|z| \dist(z, \mathds{R}_+)} \right) \frac{1}{|\gamma(L_{\hat{E},\beta}) - z| }.\label{undir_bounding_eq_I}
\end{align}

For the remaining second term, we find 
\begin{align}
	&\| [-z(L_{\hat{E},\beta} -z)^{-1}( L_{E \setminus \hat{E}} P - z I)^{-1}  -  P ( L_{E \setminus \hat{E}} P - z I)^{-1}    \|\\
\leq& 
\left(\frac{1}{\dist(z, \mathds{R}_+)} + \frac{1}{|z|} + \frac{\|QLP\|}{|z| \dist(z, \mathds{R}_+)} \right)
\cdot \frac{|z|}{|\gamma(L_{\hat{E},\beta}) - z| }.\label{undir_bounding_eq_II}
\end{align}
Thus in total we have indeed found:$ 
	\|(L_\beta - z I_{\ell^2(V)})^{-1} -  J^\uparrow  (\underline{L} - z I_{\ell^2(\underline{V})})^{-1} J^\downarrow\| 
\lesssim$ \\
\noindent$	\left\|( L_{\hat{E},\beta} - z)^{-1} - P/(-z) \right\|  \lesssim
	1/\gamma(L_{\hat{E}, \beta})
$.
\end{proof}

\medskip

\begin{Cor}\label{finite_dim_cor}
	From the above arguments we find in the setting of (\ref{prototypical_example}) that if there is only a finite number of clusters $\{E_i\}_{i=1}^N$ being considered, then since $\gamma(L_{\hat{E},\beta}) = \beta \cdot \gamma(L_{\hat{E},1})$ and we know that by finiteness $\gamma(L_{\hat{E},1}) > 0$, we have
\begin{align}
	\|
	(L_{E\setminus \hat{E}} + \beta  L_{ \hat{E}}  - z)^{-1} - J_\uparrow (\underline{L} - z)^{-1}J_\uparrow
	\|
	\lesssim 1/\beta.
\end{align}
\end{Cor}

Let us consider some examples to illustrate these results:

\begin{Ex}
In Example \ref{one_edge_div}, only one cluster (namely the edge $b \leftrightarrow c$)  is being considered. Thus Corollary \ref{finite_dim_cor} applies, and we also have \emph{norm} resolvent convergence, with norm-convergence speed $\mathcal{O}(\beta^{-1})$.
\end{Ex}

\begin{Ex}
In the setting of Example \ref{zero_is_accumulation_point}, the perturbing Laplacian $L_{\hat{E},\beta} = \beta \cdot L_{\hat{E},1}$ associated to the graph structure $(V,\hat{E})$, with edge weights $a(\nu_{2n-1},\nu_{2n}) =0$ and $a(\nu_{2n},\nu_{2n+1}) =\beta/(2n)$ has zero as an accumulation point in its spectrum. Thus $0 \in \sigma(L_{\hat{E},1})$ is not isolated and hence $0 = \gamma(L_{\hat{E},\beta}) = \beta \cdot \gamma(L_{\hat{E},1}) $. Thus we do not have \emph{norm}-reslvent convegence, but only \emph{strong}-resolvent convergence.
\end{Ex}

\begin{Ex}\label{clique_connectivity_example}
	  Theorem \ref{undir_appr_thm_strongly} is formualted in terms of the spectral gap of the perturbing Laplacain $L_{\hat{E},\beta}$ as opposed to edge-weights directly. Hence it allows us to consider settings of increasing connectivity where edge weights remain bounded. Instead we might increase connectivity by increasing the \emph{number} of edges inside clusters:
	 
\noindent	\begin{minipage}[b]{0.55\textwidth}
\noindent To this end, let us consider a graph $G_{N}$ with node set $V_N = \{\mu\} \cup \{\nu_i\}_{i = 1}^N$. We choose the corresponding edge set as $E_N = \hat{E}_N \cup \{(\mu,\nu_1), (\nu_1,\mu)\}$, with $\hat{E}_N$  so that $(V, \hat{E}_N)$ is exactly the $N$-clique $K_N$ (c.f. Fig. \ref{fig:a-k8}). All edge weights are chosen equal to one. We set the weight of node $\mu$ as $,m(\mu) = 1$ and the remaining weights as $m(\nu_n) = 1/N$ for $n \leq N$.  We now take $N \rightarrow \infty$ and compare the graph $G_N$ with the coarsified graph $\underline{G}$, where the clique $K_N$ is collapsed into a single node '$\nu$'. 
\end{minipage}
	\begin{minipage}[b]{0.5\textwidth}
	\centering
\begin{tikzpicture}[
	scale=2,
	every node/.style={circle, draw, minimum size=15pt, inner sep=0pt},
	labelnode/.style={draw=none, fill=none, inner sep=1pt}
	]

	\def\radius{0.5}
	\def\gap{1.2}

	\foreach \i in {1,...,8} {
		\node (v\i) at ({\radius * cos(180 + 360/8 * (\i - 1))}, {\gap + \radius * sin(180 + 360/8 * (\i - 1))}) {};
	}

	\node at (v1) {$\nu_1$};

	\node (a) at ({-\radius - 1}, \gap) {$\mu$};

	\node[labelnode] at ($(a)+(0,0.4)$) {(a) $G$};

	\draw (a) -- (v1);

	\foreach \i in {1,...,8} {
		\foreach \j in {1,...,8} {
			\ifnum\i<\j
			\draw (v\i) -- (v\j);
			\fi
		}
	}
	
	\node (nu) at (0,0) {$\nu$};

	\node (mu) at ({-\radius - 1}, 0) {$\mu$};

	\node[labelnode] at ($(mu)+(0,0.4)$) {(b) $\underline{G}$};

	\draw (mu) -- (nu);
	
\end{tikzpicture}

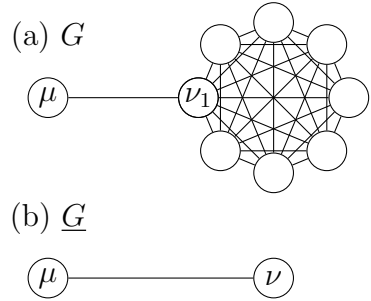
\captionof{figure}{$G$ with clique $K_8$ and reduced $\underline{G}$}
\label{fig:a-k8}
\end{minipage}

\noindent That is, the vertex set of $\underline{G}$ is
 given by $\underline{V} = \{a,\nu\}$. All node masses are set to one and the only remaining edge weight is $a(\mu, \nu) = 1$.
  Since the first non-trivial eigenvalue of an $N$-clique is given as $\lambda_1(K_N) = N$, we here hence find $\gamma(L_{\hat{E}}) = \lambda_1(K_N) = N$. Thus
\begin{align}\label{clique_implosion}
	\|  (L_N - z)^{-1} -     J_\uparrow  (\underline{L} - z)^{-1} J_\downarrow  \| \lesssim 1/N,
\end{align}
\noindent and we have generalized convergence of $L_N$ towards the Laplacian $\underline{L}$ on the graph $\underline{G}$ of Fig. \ref{fig:a-k8} (b),  as $N \rightarrow \infty$.
\end{Ex}

\begin{Rem}
We may here now note that a reasonable way to measure the connectivity in the set of clusters determined by $\hat{E}$ is to quantify the extent of the  difference '$(L_{\hat{E},\beta} -z)^{-1} - P/(-z)$'. If this difference is small, connectivity is high; and vice versa.
 Indeed, in the setting of Section \ref{undirected_strong_resolvent_convergence} this difference tends to zero strongly as edge-weights diverge. If it tends to zero in norm,  Lemma \ref{conv_to_projection} ensures, that  the traditional measure $\gamma(L_{\hat{E},\beta})$ of connectivity in the undirected graph  $G_{\beta} = (V, m, \hat{E}, a_\beta)$ diverges. Finally, as 	Example \ref{clique_connectivity_example} showcases,  the difference $\|(L_{\hat{E},\beta} -z)^{-1} - P/(-z)\| \lesssim 1/N$ also  captures  settings outside the case of divergent edge weights.
\end{Rem}

Theorem \ref{undir_appr_thm} also allows us to sharpen our heat semigroup convergence results of Corollary \ref{norm_dynamical_convergence_strong} above:
	\begin{Cor}\label{norm_dynamical_convergence}
	For any $t > 0$, we have $\|e^{-tL_\beta} - J_\uparrow e^{-t\underline{L}} J_\downarrow \| \lesssim f(\gamma(L_{\hat{E},\beta}) )$, with the function $F(\cdot)$ bounded continuous and decaying: $F(\gamma(L_{\hat{E},\beta})) \rightarrow 0$ as $\gamma(L_{\hat{E},\beta}) \rightarrow  \infty$.
	\end{Cor}
\begin{proof}
By means of the holomorphic functional calculus, we may write
\begin{align}
	e^{-tL} - J_\uparrow e^{-t\underline{L}} J_\downarrow = \frac{1}{2 \pi i } \oint_{\Gamma} e^{-tz} \left( 	(L_{E\setminus \hat{E}} + \beta  L_{ \hat{E}}  - z)^{-1} - J_\uparrow (\underline{L} - z)^{-1}J_\uparrow \right)dz.
\end{align}
for a suitable path $\Gamma$ encircling $[0, \infty) \subseteq \mathds{C}$. By the proof of Theorem \ref{undir_appr_thm}, we may bound $\|(L_{E\setminus \hat{E}} + \beta  L_{ \hat{E}}  - z)^{-1} - J_\uparrow (\underline{L} - z)^{-1}J_\uparrow \|$ by the sum of (\ref{undir_bounding_eq_I}) and (\ref{undir_bounding_eq_II}). Let us fix $\Gamma$ to be 
 to be a rectangle through the points $\{i,-1,-i, \infty\}$ c.f. Fig. \ref{int_und}. Then
\begin{align}
	&\|	e^{-tL} - J_\uparrow e^{-t\underline{L}} J_\downarrow\| \\
	\leq & \frac{1}{2 \pi} \oint_\Gamma \left(\frac{2 |z|^2}{\dist(z, \mathds{R}_+)} +1\right) \left(\frac{1}{\dist(z, \mathds{R}_+)} + \frac{1}{|z|} + \frac{\|QLP\|}{|z| \dist(z, \mathds{R}_+)} \right) \frac{|e^{-tz}|}{|\gamma(L_{\hat{E},\beta})  - z | }dz\\
	\lesssim & \oint_\Gamma \frac{1 + |z|^2}{|\gamma(L_{\hat{E},\beta})  - z |} e^{-t\Re(z)}dz
	\lesssim \frac{1}{\gamma(L_{\hat{E},\beta})} + 2\int_{-1}^\infty \frac{1 + (1 + x^2)^2}{\sqrt{(\gamma(L_{\hat{E},\beta}) - x )^2 +1 }} e^{-tx}dx.
\end{align}
 We may thus choose $F(\gamma(L_{\hat{E},\beta}))$ as the last expression above multiplied by a suitably large constant. By dominated convergence,  $F(\gamma(L_{\hat{E},\beta})) \rightarrow 0$ as $\gamma(L_{\hat{E},\beta}) \rightarrow  \infty$.
\end{proof}

	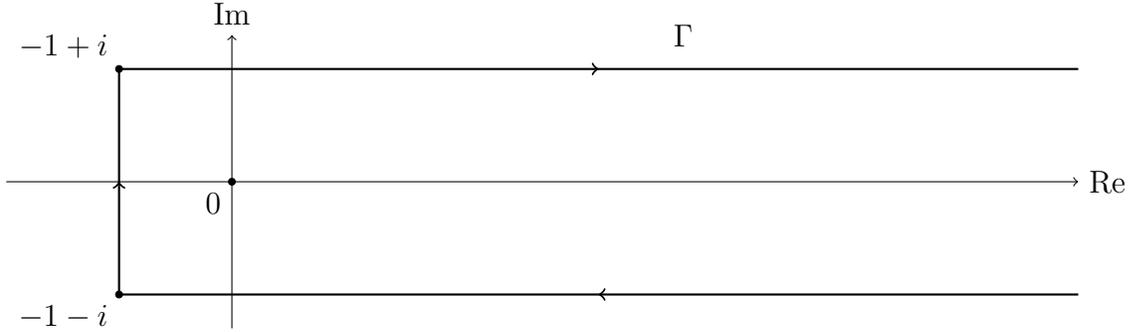
\begin{figure}[h!]
\begin{tikzpicture}[scale=1.5]
	\draw[->] (-2,0) -- (7.5,0) node[right] {Re};
	\draw[->] (0,-1.3) -- (0,1.3) node[above] {Im};
	
	\filldraw (0,0) circle (0.03) node[anchor=north east] {0};
	
	\coordinate (A) at (-1,-1);
	\coordinate (B) at (-1,1);
	\coordinate (C) at (7.5,);
	\coordinate (D) at (7.5,-1);

	\draw[thick, postaction={decorate},
	decoration={markings, mark=at position 0.5 with {\arrow{>}}}
	] (A) -- (B);
	
	\draw[thick, postaction={decorate},
	decoration={markings, mark=at position 0.5 with {\arrow{>}}}
	] (B) -- (C);

	\draw[thick, postaction={decorate},
	decoration={markings, mark=at position 0.5 with {\arrow{>}}}
	] (D) -- (A);
	
	\filldraw (A) circle (0.03) node[anchor=north east] {$-1 - i$};
	\filldraw (B) circle (0.03) node[anchor=south east] {$-1 + i$};

	\node at (4,1.3) {$\Gamma$};
\end{tikzpicture}
	\captionof{figure}{Integration path $\Gamma$}
\label{int_und}
\end{figure}

	\section{Convergence under increasing connectivity II: Directed Graphs}\label{directed_convergence}
	In the undirected setting, we saw in the proofs of Theorems \ref{undir_appr_thm_strongly} \&  \ref{undir_appr_thm}, that the sequence of Laplacians $L_{E \setminus \hat{E}} + L_{\hat{E},\beta}$ converged in a generalized resolvent sense towards the limit operator $PL_{E \setminus \hat{E}}P$. Here $P = \chi_{\{0\}}( L_{\hat{E},\beta})$ is the spectral projection onto the kernel of $ L_{\hat{E},\beta}$, which is independen of $\beta$. We then canonically decomposed this projecetion as $P  = J_\uparrow J_\downarrow$ and identified the limit operator with a Laplacian $\underline{L}$ on a coarse-grained graph as $\underline{L} =  J_\downarrow PL_{E \setminus \hat{E}}PJ_\uparrow  =  J_\downarrow L_{E \setminus \hat{E}}J_\uparrow  $.
	
	\medskip
	
		In the directed setting, the situation is similar in spirit, but technically more involved: Since (in- and out-degree) Laplacians are generically no longer even self adjoint, let alone positive semi-definite, we for example  lose access to monotone convergence theorems 
	 as generic tools to establish strong convergence under assumptions as weak as merely demanding $a_\beta(i,j) \rightarrow \infty$ for all $(i,j) \in \hat{E}$.

\subsection{Large coupling convergence and the Riesz projector}

However, the situation is not hopelessly lost:	In the simpler setting of linear scaling and under mild well-behavedness assumptions, we still have resolvent convergence  for operator families of the form $A + \beta B$,  as $\beta \rightarrow \infty$ \cite{koke2026coupling}.  
	In this setting (where   $A, B$ are no longer self-adjoint)  the orthogonal spectral projection 
 onto the kernel of the perturbing operator $B$  (i.e. onto the kernel of $L_{\hat{E},\beta}$ in the notation of the preceeding section) is no longer well-defined. Instead, it is replaced by the somewhat more general
	 \emph{Riesz projector} (c.f. e.g. \cite{HislopSigal1996, koke2026coupling}). This projection  is canonically associatied to the point  $0 \in \sigma(B)$ via
	\begin{align}\label{riesz_projector}
		P:= \frac{1}{2 \pi i} \oint_{S_\epsilon^1} (B - z I)^{-1} \, dz.
	\end{align}
 For the Riesz projector to be well defined,   $0 \in \sigma(B)$ needs to be an isolated point in the spectrum of $B$. Precisely then  there exists $\epsilon > 0$ for which the cicle $S_\epsilon^1$ of radius $\epsilon$ lies completely within the resolvent set of $B$ and exactly encircles zero  and no other points of the spectrum  $\sigma(B)$.
The projection (\ref{riesz_projector}) is generically not orthogonal. If $B$ however is self-adjoint  (and $0 \in \sigma(B)$ is isolated), it exactly coincides with the spectral projection (as e.g. calculated via the Borel functional calculus) onto onto $\ker(B)$ \cite{HislopSigal1996}.
\medskip

 In this non-self-adjoint-setting, for  $A,B $ bounded, the limit operator of the family $A + \beta B$ determined in \cite{koke2026coupling} is then given by $PAP$.
  To see how we may apply this convergence result
    to our graph  setting, we consider a directed graph, and linearly perturb edges in a finite set $\tilde{E} \ni (i,j)$:
 
 	\begin{Def}\label{directed_clusters} Given a graph $G$, a collection of directed clusters $\{\tilde E_i\}_{i \in I}$ is a finite family  of edge sets ($\tilde E_i \subseteq E$) with total set $\tilde{E} = \cup_{i \in I} \tilde{E}_i$ so that 
 	\begin{enumerate}[label=(\roman*), topsep=0pt, itemsep=-1em]
 		\item Each $\tilde E_i$ is finite.\\ 
 		\item Each $\tilde E_i$ determines a single reach of $(V,\tilde{E})$.
 	\end{enumerate}
 	Associated to each cluster $\tilde E_i$ is the corresponding set $\tilde{V}_i =  \{v| \exists a \in V: (v,a) \in \tilde E_i \lor (a,v) \in \tilde E_i \}$ of nodes with at least one edge  in the respective cluster $\tilde{E}_i$ incident at or emanating at them (i.e. the nodes in the reach determined by $\tilde E_i$).
 \end{Def}
 \begin{Rem}
 Note that given the total set $\tilde{E}$, we can always uniquely decompose it into its reaches $\tilde{E}_i$. It should also be noted that in contrast to the undirected setting, we here no longer assume that the sets $\tilde{E}_i$ are pairwise \emph{node}-disjoint (i.e. in general $V_i \cap V_j \neq \emptyset$). This difference arises because while in the undirected setting distinct connected components are necessarily disjoint, \emph{reaches} in the directed setting may share nodes among them.
  \end{Rem}

 \medskip
 
 On the set $\tilde{E}$, we then perturb edge weights as  $a_\beta(i,j) = \beta a(i,j)$, with $\beta \gg 1$.
 Hence we may write 
 $
 	\Lio_\beta = \Lio_{E \setminus \tilde{E}} + \beta \Lio_{\tilde{E}}
$.
 Thus \cite[Theorem 3.12]{koke2026coupling}
  implies\footnote{To apply this Theorem, we also need $P^\pm \Lio_{\tilde{E}}P^\pm = 0$. This holds,  as discussed in Section \ref{chap:kernels} below.} that we have the norm resolvent convergence 
 \begin{align}\label{initial_dir_conv}
 	\|(\Lio_\beta - z)^{-1} - P^\pm (P^\pm \Lio_{E \setminus \tilde{E}}P^\pm - z)^{-1}  P^\pm \| \lesssim 1/\beta \xrightarrow{\beta \rightarrow \infty} 0.
 \end{align}
Here $P^\pm$ is the Riesz projection onto the kernel of $\Lio_{\tilde{E}}$. This projection is well defined via (\ref{riesz_projector}), since we assumed that $\tilde{E}$ is a finite set, which
ensures that  $L_{\tilde{E}}$ is  finite dimensional, so
that $0\in \sigma(L^\pm_{\tilde{E}})$ is an isolated point.

\medskip 

 In order to cast (\ref{initial_dir_conv}) into graph theoretic terms and generalize it further, we have to better understand the Riesz projector $P^\pm$ and its relationship to the (left- and right-) kernels of $\Lio$. With $\tilde{V} = \cup_i V_i$, we note the following:

	\begin{Thm}\label{riesz_demystified}
			The Riesz projections $P^\pm$ onto the kernels of $\Lio_{\hat E}$ are given as
		\begin{align}
			P^- &= \chi_{V \setminus \tilde{V}} +  \sum\limits_{R \in \mathfrak{R}}       \vec \gamma^{-}_{R}         \langle \cev \gamma^{-}_{R}, \cdot \rangle_{\ell^2(G)}, \quad
						P^+ =  \chi_{V \setminus \tilde{V}} + \sum\limits_{R^{\intercal} \in \mathfrak{R}^{\intercal}}       \vec\gamma^{+}_{R^{\intercal}}    
			\langle \cev\gamma^{+}_{R^{\intercal}}, \cdot \rangle_{\ell^2(G)}.
		\end{align}
		Here we have used the notation of Section \ref{directed_kernels} and e.g. denoted by $\mathfrak{R}^\intercal$ the set of reaches in $G^\intercal$. 
		\end{Thm}
		Note that multiplication with the  function $ \chi_{V \setminus \tilde{V}}$ simply implements the identity operation on any node outside $\tilde{V} = \cup_i V_i$: For any $f$ completely supported on nodes outside  $\tilde{V} $, we have $P^\pm f = \chi_{V \setminus \tilde{V}} f = f$. Before providing the proof of the above theorem, we first establish
		that the nilpotent part of $\Lio_{\hat{E}}$ at $0 \in \sigma(\Lio_{\hat{E}})$ vanishes:
		\begin{Lem}\label{vanishing_nilpotent} Let $P^\pm$ be the Riesz projector onto $0 \in \sigma( \Lio_{\tilde{E}})$.
			We have $P^\pm \Lio_{\tilde{E}} =  \Lio_{\tilde{E}}P^\pm = P^\pm \Lio_{\tilde{E}}P^\pm = 0$.
		\end{Lem}
		\begin{proof}
			This statement is equivalent to saying that the algebraic- and geometric multiplicities of $0 \in \sigma(\Lio_{\tilde{E}}{\restriction_{\tilde{V}})}$  coincide, which is true by  \cite[Theorem 3.2]{caughveer}).
			Hence
			the nilpotent part  in the Jordan Chevalley decomposition of 
			$\Lio_{\tilde{E}}{\restriction_{\tilde{V}}}$ at zero vanishes
			\cite{CoutyEsterleZarouf2011, kato1995perturbation}.
		\end{proof}

	\begin{proof}[Proof of Theorem \ref{riesz_demystified}]
		By Lemma \ref{vanishing_nilpotent} the eigenvalue $0 \in \sigma( \Lio_{\tilde{E}})$ is semi-simple, so that the Riesz projector is exatctly a projection onto the left- and right kernels of $L^\pm_{\tilde{E}}$ \cite{kato1995perturbation}. We establish that this projection is exactly given by $P^\pm$:
		Using $\langle \cev\gamma^{-}_{R}, \vec\gamma^{-}_{S}\rangle_{\ell^2(G)} = \delta_{RS}$ and  $	\langle \cev\gamma^{+}_{R^\intercal}, \vec\gamma^{+}_{S^\intercal}\rangle_{\ell^2(G)} = \delta_{R^\intercal S^\intercal}$ 
		as guaranteed by Propositions  \ref{prop:leftkernel_-} and   \ref{prop:leftkernel_+} respecively, it is easilly verified that  $P^\pm \cdot P^\pm = P^\pm$,  so that the $P^\pm$ are indeed projections. Similarly one verifies that $P^\pm$ preseves the corresponding left- and right-kernels of $\Lio_{\tilde{E}}$. 		 Clearly the dimension of the range of $P^\pm$ agrees with the dimension of the (left- and right-) kernel of $\Lio_{\hat{E}}$, so that the claim is proved.
		\end{proof}

	\subsection{Norm-resolvent convergence to Laplacians on reduced graphs}
	
Given a collection of clusters as in Definition \ref{directed_clusters}, we are now interested in establishing norm resolvent convergence  of the familiy $
	\Lio_\beta = \Lio_{E \setminus \tilde{E}} + \beta \Lio_{\tilde{E}}$ towards reduced  Laplacians $\underline{\Lio}$ on reduced graphs $\underline{G^\pm}$ as $\beta \rightarrow \infty$. We separate the discussion into results concerning the in-degree  ($-$) setting and the out-degree ($+$) setting.

\subsubsection{In-degree Laplacian}
For families of the form $	\Lin_\beta = \Lin_{E \setminus \tilde{E}} + \beta \Lin_{\tilde{E}}$, we will establish convergene towards a limit Laplacian defined on a reduced graph $\underline{G}^-$ as $\beta \rightarrow \infty$. We begin by defining this reduced limit-graph in the in-degree setting:

	\begin{Def}\label{indeg_coarse_G_def} Given a graph $G = (V, m, E, a)$ and a countable collection $\{\tilde{E}_i\}_{i \in I}$ of clustered edges,
		we define a reduced graph $\underline{G}^- = \{\underline{V}^-, \underline{m}^-, \underline{E}^-,\underline{a}^-\}$ as follows:
		Define an equivalence relation "$\sim_{\tilde{E}}$" on the node-set $V$ by setting $a \sim_{\hat{E}} b $ iff
		$a$ and $b$ are contained in the same reach of $(V, \tilde{E})$.
		\begin{enumerate}[label=(\roman*), topsep=0pt, itemsep=-.2em]
			\item 	The reduced node set $\underline{V}^-$ is defined as the set of reaches of $(V, \tilde{E})$:
			\begin{align}
				\underline{V}^- = V/ \sim_{\tilde{E}}.
			\end{align}
			\item The weight of a node $R \in \underline{V}$ is given as
			\begin{align}\label{in_deg_mass}
				\underline{m}^-(R) := \langle  \chi_R ,P^-  \chi_R  \rangle_{\ell^2(G)}
			\end{align}
			\item The edge-weight function $\underline{a}:\underline{V} \times \underline{V} \rightarrow [0, \infty)$ is defined similarly as:
			\begin{align}\label{weight_def_-}
				\underline{a^-}(R, S) := \underline{m}^-(R)\cdot\sum\limits_{r \in R, s \in S} 
				[
				\cev{\gamma}^{-}_{R}(r)
				a(r,s)
				\vec{\gamma}_S^{-}(s)
				]
			\end{align}
			\item The set of edges $\underline{E}$ is given as the support of $\underline{a}$ on $\underline{V} \times \underline{V}$:
			\begin{equation}
				\underline{E}^- = \{(R,S)\in\underline{V}\times\underline{V}: \underline{a}^-(R,S) >0 \}.
			\end{equation}
		\end{enumerate}
	\end{Def}
	
	\medskip
	
A few comments are in order: To understand the mass definition (\ref{in_deg_mass}) better, we note
 that we have
$
\underline{m}^-(R) := \langle  \chi_R ,P^-  \chi_R  \rangle_{\ell^2(G)}
			=	\langle  \chi_R , \vec\gamma^{-}_{R}   \rangle_{\ell^2(G)}  
			\cdot
				\langle \cev\gamma^{-}_{R}, \chi_R \rangle_{\ell^2(G)}
			=	\langle  \chi_R , \vec\gamma^{-}_{R}   \rangle_{\ell^2(G)}  
$.
In the undirected case, we have $\vec\gamma^{-}_{R} = \chi_R$, so that $	\langle  \chi_R , \vec\gamma^{-}_{R}   \rangle_{\ell^2(G)}  = \|\chi_R\|^2_{\ell^2(G)} =  \sum_{v \in R} m(v)$. Hence  (\ref{in_deg_mass}) coincides with our earlier definition  (\ref{mass_def_undirected}) of aggregated mass in the undirected setting (c.f. also Rem \ref{naturality_remark_undirected}). 
 In the generic directed setting, we have,  with the exclusive ($H(R)$) and common ($C(R)$) parts of $R$, that  
	\begin{align}
		\underline{m}(R) = \sum\limits_{r \in R} m(r)  \vec\gamma^{-}_{R}(r) = \sum\limits_{r \in H(R)} m(r) + \sum\limits_{r \in C(R)} m(r)   \vec\gamma^{-}_{R}(r).
	\end{align}
Thus (\ref{in_deg_mass}) extends our earlier undirected definition (\ref{mass_def_undirected}) by accounting for the fact that nodes may now belong to more than one reach. Indeed, (\ref{in_deg_mass}) avoids an 'overcounting' of node masses when aggregating, and ensures that mass is conserved when going from $G$ to $\underline{G}^-$: If the mass $M = \sum_{v \in V} m(v)$ of our original graph $G$ is finite, we have with $\underline{M}^- = \sum_{R \in \underline{V}} \underline{m}^-(R)$ the total mass of $\underline{G}$, that
	\begin{equation}
		\underline{M}^- = \sum\limits_{R \in \underline{V}} \underline{m}^-(R) 
		=
		 \sum\limits_{R \in \underline{V}}
		 \left(
		 \sum\limits_{r \in H(R)} m(r) + \sum\limits_{r \in C(R)} m(r)   \vec\gamma^{-}_{R}(r)
		 \right) 
 		= \sum\limits_{r \in V} m(r)
		= M.
	\end{equation}
	The third equality follows from the fact that the $\{\vec\gamma^{-}_{R}\}_{R \in \mathfrak{R}}$ form a partition of unity (i.e. $	\sum_{R \in \mathfrak{R}}\,	\vec\gamma^{-}_{R}(j) = 1$ $\forall i \in V$),
	as ensured by Proposition \ref{right_kernel_Lin}.	
	
		\medskip

For the definition (\ref{weight_def_-}) of aggregated edge-weigths, we also note that in the undirected setting it exactly reduces to our earlier definition (\ref{edge_def_undir}): In the undirected setting, we have $	\vec{\gamma}^{-}_{S} = \chi_S$, $	\cev{\gamma}_R^{-}= \chi_R/(\sum_{v \in R} m(r))$, and $\underline{m}^{-}(R) = \sum_{v \in R} m(r)$. Thus
		\begin{align}
	\underline{a^-}(R, S) 
	= \frac{ \underline{m}^-(R)}{ \underline{m}^-(R)}	\sum\limits_{r \in R, s \in S} 
\chi_{R}(r)
	a(r,s)
\chi_S(s) 
= 
	\sum\limits_{r \in R, s \in S} 
a(r,s).
\end{align}

\medskip

	We next decompose the Riesz projection of Theorem \ref{riesz_demystified}   as $
	 	P^- = J^-_{\uparrow}J^-_\downarrow$,
	 with $J^-_\downarrow$ mapping to the Hilbert space $\ell^2(\underline{G}^-)$ and $J^-_{\uparrow}$ mapping back up to $\ell^2(G)$:
	 \begin{Def}\label{proj_interp_indeg}
	 	At the level of Hilbert spaces, we define the bounded projection $J^-_\downarrow: \ell^2(G) \rightarrow \ell^2(\underline{G}^-)$ and interpolation $J^-_\uparrow: \ell^2(\underline{G}^-) \rightarrow \ell^2(G)$ operators as 
	 	\begin{align}
	 		[J^-_\downarrow f](R) = \langle \cev\gamma^{-}_{R} , f \rangle_{\ell^2(G)},\quad \text{and} \quad J^-_\uparrow \underline{f} = \sum_{R \in \mathfrak{R}} \underline{f}(R)  \vec\gamma^{-}_{R}.
	 	\end{align}
	 In comparison to Definition \ref{proj_interp_undirected}  we have here absorbed the mass $\underline{m}^-(R)$ into  $\cev\gamma^{-}_{R}$, so that it no longer appears explicitly, as in it did in the undirected setting (\ref{undirected_Js}). 
	 \end{Def}
	 
	 As in the undirected setting, there in principle remains one scalar degree of freedom per node in Definition \ref{proj_interp_indeg}, as $P^-$ is preserved under the simultaneous changes 
	 $[J^-_\downarrow f](R) \mapsto b_R \langle \cev\gamma^{-}_{R} , f \rangle_{\ell^2(G)}$, $J^-_\uparrow \underline{f} \mapsto \sum_{R \in \mathfrak{R}} \underline{f}(R)  \vec\gamma^{-}_{R}/b_R$ for all $b_R\neq 0$ Setting $b_R \equiv 1$ as in Definition \ref{proj_interp_indeg} above however is the natural choice:
	 \begin{Rem}\label{norm_preservation_rem_-}
	Here in the directed setting, $J^-_\uparrow$ as in Definition \ref{proj_interp_indeg} generically no longer is an isometry. However, it still maps probability distributions to probability distributions: Indeed, let $f \in \ell^1(G) \cap \ell^2(G)$ be such that $f: V \rightarrow   [0,\infty )$ and 
	\begin{align}
	\sum_{v \in V} f(v) m(v)  =	\sum_{v \in V} \langle f, e_v \rangle_{\ell^2(G)} =  1.
	\end{align}
Define a new function $\check{f}: \underline{G}^- \rightarrow [0, \infty)$ via
$\langle \check{f},\cdot \rangle_{\ell^2(\underline{G}^-)} =  \langle f, J_\uparrow^- \cdot \rangle_{\ell^2(G)}$.\footnote{While this definition might seem counter-intuitive at first, it is related to the fact that $e^{-t \Lin}$ preserves probabilty distributions "$p$" only if it acts as $\langle p, \cdot \rangle \mapsto \langle p,e^{-t \Lin} \cdot \rangle $ \cite[Section 6]{Veerman2020APO}.} Then
\begin{align}
	\sum\limits_{R \in \underline{V}^-} \check{f}(R) \underline{m}^-(R) = 	\sum\limits_{R \in \underline{V}^-}  \langle \check{f}, e_R \rangle_{\ell^2(\underline{G}^-)}
=
\sum\limits_{R \in \underline{V}^-}  \langle f, J_\uparrow e_R \rangle_{\ell^2(G)}
= \sum\limits_{R \in \underline{V}^-}  \langle f, \vec\gamma^{-}_{R} \rangle_{\ell^2(G)} 
\end{align}
By Proposition  \ref{right_kernel_Lin},  $\sum_{R \in \mathfrak{R}}\,	\vec\gamma^{-}_{R}(j) = 1$, so that $\sum_{R \in \underline{V}^-}  \langle f, \vec\gamma^{-}_{R} \rangle_{\ell^2(G)}  = \sum_{v \in V} f(v) m(v)$.

\medskip A similar statement is true for $J^-_\downarrow$: Let $\underline{f} \in \ell^2(\underline{G}^-)\cap \ell^1(\underline{G}^-) $ be a probability distribution (i.e. $\underline{f}: \underline{G} \rightarrow [0, \infty)$ and $\sum_{R \in \underline{V}^-} \underline{f}(R) \underline{m}^-(R)  = 1$). Then define a function $\hat{\underline{f}}$ on $G$ via 
$\langle \hat{\underline{f}} , \cdot \rangle_{\ell^2(G)}  
=   \langle \underline{f}, J^-_\downarrow   \cdot \rangle_{\ell^2(\underline{G}^-)} $. Then $\hat{\underline{f}}$ is again a probability distribution:
\begin{align}
&\sum_{v \in V}	\hat{\underline{f}}(v) m(v) = \sum_{v \in V}	\langle \hat{\underline{f}}(v), e_v \rangle_{\ell^2(G)} 
= 
\sum_{v \in V} \langle \underline{f}, J^-_\downarrow   e_v \rangle_{\ell^2(\underline{G}^-)}\\
 =
&\sum_{R \in \underline{V}^-} 
 \langle \underline{f}, e_R \rangle_{\ell^2(\underline{G}^-)}
 \cdot 
 \sum_{v \in V} 
\langle e_R, J^-_\downarrow   e_v \rangle_{\ell^2(\underline{G}^-)}/\underline{m}^-(R) 
= \sum_{R \in \underline{V}^-} 
\langle \underline{f}, e_R \rangle_{\ell^2(\underline{G}^-)} = 1
\end{align}
Here we have noted $ \sum_{v \in V} 
\langle e_R, J^-_\downarrow   e_v \rangle_{\ell^2(\underline{G}^-)}/\underline{m}^-(R) = 1$, since by Prop. \ref{prop:leftkernel_-} we have $\sum_{j}\, \langle \cev \gamma^{-}_{R}, e_j \rangle =1 $ and by Definition \ref{proj_interp_indeg}, we have $J^-_\downarrow  = \sum_{R \in \underline{V}} e_R \langle \cev\gamma^{-}_{R} , \cdot \rangle_{\ell^2(V)}$.
\end{Rem}

	 \medskip 
	To the graph $\underline{G}^-$ of Definition \ref{indeg_coarse_G_def} we can then canonically associate a graph Lapalcian $\underline{L}^-$ via  Definition \ref{laplacian_definition}. As we show now, this is the limit Laplacian, to which the family $\Lin_\beta = \Lin_{E \setminus \tilde{E}} + \beta \Lin_{\tilde{E}}$ converges:
\begin{Thm}\label{-_conv_theorem}
	We have
$
		\| (\Lin_\beta - z)^{-1} - J^-_{\uparrow} (\underline{L}^- -z)^{-1}  J^-_\downarrow \|
		\lesssim 1/\beta \xrightarrow{\beta \rightarrow \infty} 0.
$
\end{Thm}
In order to prove this result we first establish an auxiliary Lemma:
\begin{Lem}\label{aux_lemma_-}
Let $\Lin_{\hat{E}}$	be an in-degree Laplacian with finite range. Denote by $P^- = J^-_\uparrow J^-_\downarrow$ the Riesz projection onto $0 \in \sigma(L_{\hat{E}})$. We have
$
	\|  (\beta \Lin_{\tilde{E}} - z)^{-1} - P^-/(-z)  \| \lesssim 1/\beta 
$
\end{Lem}
\begin{proof} 
	Since $\tilde{E}$ is a finite set, $0 \in \sigma(\Lin)$ is isolated, and the Riesz Projector onto this point in the spectrum is well-defined.
	With $Q^- = Id - P^-$, we may thus write
	$L_{E \setminus \hat{E}} = P^- L_{E \setminus \hat{E}} P^- + Q^- L_{E \setminus \hat{E}} Q^-$.  By
	Lemma \ref{vanishing_nilpotent} we have
	$P^- L_{E \setminus \hat{E}} P^- = 0$. Thus block inversion, yields
	\begin{align}
		(\beta L_{E \setminus \hat{E}}  -z I_{\ell^2(G)})^{-1}
		=
		\begin{pmatrix}
			- \frac{1}{z} P^-  & 0 \\
			0 & (\beta Q^- \Lin_{E \setminus \hat{E}}  Q^- - zQ^-)^{-1} 
		\end{pmatrix}.
	\end{align}
	Thus we have 
	\begin{align}
		\|  (\beta \Lin_{\tilde{E}} - z)^{-1} - P^-/(-z)  \| = \|Q^-( Q^- L_{E \setminus \hat{E}}  Q^- - zQ^-/\beta)^{-1}Q^- \|/\beta \leq 2\| [Q^- L_{E \setminus \hat{E}}  Q^-]^+\|/\beta.
	\end{align}
	Here '$^+$' denotes the Penrose pseudo-inverse, and the last estimate follows by a Neumann argument similar to the one in the proof of Theorem \ref{undir_appr_thm}.
\end{proof}

\begin{proof}[Proof of Theorem \ref{-_conv_theorem}]
	Using Lemma \ref{aux_lemma_-}, it suffices to establish 
	\begin{align}
		\| (\Lin_\beta - z)^{-1} - J^-_{\uparrow} (\underline{L}^- -z)^{-1}  J^-_{\downarrow} \|
		\leq C \| (\beta \Lin_{\tilde{E}} -z)^{-1} - P^-/(-z)\| 
	\end{align}
	to prove Theorem  \ref{-_conv_theorem}.
	We again first establish resolvent norm-closeness to $P^- \Lin_{E \setminus \tilde{E}} P^-$, and then show $P^- \Lin_{E \setminus \tilde{E}} P^- = J^-_{\uparrow} \underline{L}^- J^-_{\downarrow}$.

	\medskip

	By the same line of reasoning as in the proof of Theorem \ref{undir_appr_thm}, we find:	 
	\begin{align}
		&\left\| (\Lin_{E \setminus \tilde{E}} + \beta \Lin_{\tilde{E}} - z I)^{-1} - P^- ( \Lin_{E \setminus \tilde{E}} P^- - z I)^{-1} P^- \right\| \\
		=& \left\| [-z(\beta \Lin_{\tilde{E}} -z)^{-1} (-z \Lin_{E \setminus \tilde{E}} (\beta \Lin_{\tilde{E}} -z)^{-1} - z)^{-1}] -  P^- ( \Lin_{E \setminus \tilde{E}} P^- - z I)^{-1} \right\| \\
		\leq &     
		\left\| [-z(\beta \Lin_{\tilde{E}} -z)^{-1} (-z \Lin_{E \setminus \tilde{E}} (\beta \Lin_{\tilde{E}} -z)^{-1} - z)^{-1}] -  [-z(\beta \Lin_{\tilde{E}} -z)^{-1}( \Lin_{E \setminus \tilde{E}} P^- - z I)^{-1} ] \right\| \\
		+& \left\| [-z(\beta \Lin_{\tilde{E}} -z)^{-1}( \Lin_{E \setminus \tilde{E}} P^- - z I)^{-1} ] -  P^- ( \Lin_{E \setminus \tilde{E}} P^- - z I)^{-1} \right\|
	\end{align}
	For the first term, we find 
	\begin{align}
		& 	\left\| \left[ -z(\beta \Lin_{\tilde{E}} -z)^{-1} \left(-z \Lin_{E \setminus \tilde{E}} (\beta \Lin_{\tilde{E}} -z)^{-1} - z \right)^{-1} \right] -  \left[ -z(\beta \Lin_{\tilde{E}} -z)^{-1}( \Lin_{E \setminus \tilde{E}} P^- - z I)^{-1} \right] \right\| \\
		\leq & \, |z| \cdot \left\| (\beta \Lin_{\tilde{E}} -z)^{-1} \right\| \cdot \left\|
		\left(-z \Lin_{E \setminus \tilde{E}} (\beta \Lin_{\tilde{E}} -z)^{-1} - z\right)^{-1} - ( \Lin_{E \setminus \tilde{E}} P^- - z I)^{-1} 
		\right\| \\
		\leq & \,|z|\left( \|P^-\|/|z|   +   2\|[Q\Lin_{\hat{E}}Q]^+\|/\beta   \right) \\
		\cdot&
		\left\|(\Lin_{E \setminus \tilde{E}} P^- -z )^{-1} \sum\limits_{k = 1}^{\infty} \left[-(\Lin_{E \setminus \tilde{E}} P^- -z )^{-1} \Lin_{E \setminus \tilde{E}}  \left[-z (\beta \Lin_{\tilde{E}} -z)^{-1}- P^-\right]\right]^k\right\| \\
		\lesssim & 	 \left\|-z (\beta \Lin_{\tilde{E}} -z)^{-1}- P^-\right\|.
	\end{align}
	For the second term, we find 
	\begin{align}
		&	\|  -z(\beta \Lin_{\tilde{E}} -z)^{-1}( \Lin_{E \setminus \tilde{E}} P^- - z I)^{-1}  
		-  P^- ( \Lin_{E \setminus \tilde{E}} P^- - z I)^{-1}    \| \\
		\leq &
		|z|	\cdot \|( \Lin_{E \setminus \tilde{E}} P^- - z I)^{-1}  \| \cdot \|  (\beta \Lin_{\tilde{E}} -z)^{-1}
		-  P^-   \|.
	\end{align}	 
	Thus convergence of $\mathcal{O}(1/\beta)$ towards $P^- \Lin P^-$ is proved.
	
	\medskip		 
	It remains to establish	that
	$P^- \Lin_{E \setminus \tilde{E}} P^- = J^-_{\uparrow} \underline{L}^- J^-_{\downarrow}$. 	This follows after noting $\underline{L}^- = J^-_{\downarrow} \Lin_{E\setminus \tilde{E}}  J^-_\uparrow$.	 Indeed:
	$
	\langle e_R, J^-_{\downarrow} \Lin_{E\setminus \tilde{E}}  J^-_\uparrow e_S\rangle_{\ell^2(\underline{G}^-)} = \langle e_R, e_R\rangle_{\ell^2(\underline{G}^-)}\cdot
	\langle\cev{\gamma}^-_R,  \Lin_{E\setminus \tilde{E}} \vec{\gamma}^-_S \rangle_{\ell^2(G)}$.
	We have $ \langle e_R, e_R\rangle_{\ell^2(\underline{G}^-)} =  \underline{m}^-(R) $.  For $\langle\cev{\gamma}^-_R,  \Lin_{E\setminus \tilde{E}} \vec{\gamma}^-_S \rangle_{\ell^2(G)}$, we find, that if $R \neq S$:
	\begin{align}
		\langle\cev{\gamma}^-_R,  \Lin_{E\setminus \tilde{E}} \vec{\gamma}^-_S \rangle_{\ell^2(G)} = - 	\sum\limits_{r \in R, s \in S} 
		[
		\cev{\gamma}^{-}_{R}(r)
		a(r,s)
		\vec{\gamma}_S^{-}(s)
		].
	\end{align}
	Thus the off diagonal elements of $J^-_{\downarrow} \Lin_{E\setminus \tilde{E}}  J^-_\uparrow $ agree with those of $\underline{\Lin}$ as associated via Definition \ref{laplacian_definition} to the graph $\underline{G}^-$ determined by Definition \ref{indeg_coarse_G_def}.
	
	\medskip
	To see that also the diagonal entries of the two operators agree, we note that for fixed $R$, we have (since by Proposition \ref{right_kernel_Lin} the $\{\vec{\gamma}_S\}_{S \in \mathfrak{R}}$ are a  partition of unity) 
	$
		\sum_{S \in \underline{V}} \langle e_R, J^-_{\downarrow} \Lin_{E\setminus \tilde{E}}  J^-_\uparrow e_S\rangle_{\ell^2(\underline{G}^-)} =  \underline{m}^-(R) \cdot \sum_{S \in \underline{V}} \langle\cev{\gamma}^-_R,  \Lin_{E\setminus \tilde{E}} \vec{\gamma}^-_S \rangle_{\ell^2(G)} = 0
$.
	Thus we indeed find $J^-_{\downarrow} \Lin_{E\setminus \tilde{E}}  J^-_\uparrow = \underline{L}^-$, since
	\begin{align}
		-	\langle e_R, J^-_{\downarrow} \Lin_{E \setminus \tilde{E}} J^-_\uparrow e_R \rangle_{\ell^2(\underline{G}^-)}  
		=  \sum_{\substack{ S \in \underline{V} \\ S \neq R}} \langle e_R, J^-_{\downarrow} \Lin_{E \setminus \tilde{E}} J^-_\uparrow e_S \rangle_{\ell^2(\underline{G}^-)} 
		=  \sum_{\substack{ S \in \underline{V} \\ S \neq R}} \langle e_R, \underline{L}^- e_S \rangle_{\ell^2(\underline{G}^-)}.
	\end{align}
\end{proof}

From the proof of the Theorem \ref{-_conv_theorem}, we may also immediately find that 
the Laplacian on $\tilde{E}$ need not diverge linearly:

\begin{Cor}
	Instead of linear scaling, consider a modification $a \mapsto a_\beta$ on $\tilde{E}$ so that for some projection $P$ onto the kernel of a fixed auxiliary Laplacian on $(V,\tilde{E})$. Suppose  the family of Laplacians $\Lin_{\tilde{E}, \beta}$ on $\tilde{E}$ satisfies $\| (\Lin_{\tilde{E},\beta} -z)^{-1} - P^-/(-z)\| \rightarrow 0$. Then
	$
	\| (\Lin_\beta - z)^{-1} - J^-_{\uparrow} (\underline{L}^- -z)^{-1}  J^-_\downarrow \|\leq C \| (\Lin_{\tilde{E},\beta} -z)^{-1} - P^-/(-z)\| \xrightarrow{\beta \rightarrow \infty} 0.
	$
\end{Cor}
It should however be noted that the condition $\| (\Lin_{\tilde{E},\beta} -z)^{-1} - P^-/(-z)\| \rightarrow 0$ is much stronger than merely demanding $a^\beta(i,j) \rightarrow \infty$ for edges $(i,j) \in \tilde{E}$, as was considered e.g. in the undirected setting of Theorem \ref{undir_appr_thm_strongly}. The underlying reason is that  the left- and right- kernels of di-graph Laplacians depend not only on the reaches of $G$ as sets, but also the specific connectivity structure inside these reaches.

Let us now consider some examples to which Theorem \ref{-_conv_theorem} applies:
\begin{Ex}
In maybe the most straightforward general setting, the reaches of $ (V, \tilde{E})$ are all node-disjoint (i.e. no reaches share common nodes). In this case, the basis of  the right-kernel of $L^-$  of Proposition \ref{right_kernel_Lin} is simply given as $\{\chi_R\}_{R \in \mathfrak{R}}$. Similarly the corresponding basis of the left-kernel of Proposition \ref{prop:leftkernel_-} is simply given as $\{\omega_R/\langle\omega_R,\chi_R\rangle\}_{R \in \mathfrak{R}}$. Masses 
 of the individual reaches are simply given as $\underline{m}^-(R) = \sum_{r \in R} m(r)$. Edge weights are accumulated as $
		\underline{a^-}(R, S) := \frac{\underline{m}^-(R)}{\langle \omega_R, \chi_R\rangle}\cdot\sum_{r \in R, s \in S} 
	\omega_R(r)
	a(r,s)
$. 
\end{Ex}

	\noindent		\begin{minipage}[b]{0.7\textwidth}
			\begin{Ex}\label{dir_two_edges_div}	Let us revisit the setting of Example \ref{one_edge_div}:
	We again consider an original graph $G$ on three nodes. However, as depicted in Fig. \ref{fig:graphs_dir}, we now split the originally undirected edges between nodes 
	into two directed edges with different weights. 
		The Laplacian $\beta \Lin_{\tilde{E}}$ is then given as 
	\begin{align}
\beta	\Lin_{\tilde{E}} =   \beta  M^{-1}		\begin{pmatrix}
			0 & 0 & 0 \\
			0&\eta & -\eta \\
			0 & -\rho & \rho
		\end{pmatrix}.
	\end{align}
	The right kernel of this operator is spanned by  $ \vec{\gamma}^{-}_{R_1} = (1,0,0)^\intercal$, $\vec{\gamma}^{-}_{R_2} =  (0,1,1)^\intercal$. A basis of the left kernel is given 
\end{Ex}
	\end{minipage}
\begin{minipage}[b]{0.4\textwidth}
	\centering
	\begin{minipage}[b]{0.6\textwidth}
		\raggedleft
		\begin{tikzpicture}
			\node[draw, circle, minimum size=12pt, inner sep=0pt] (a) at (1,1.73) {$1$};
			\node[draw, circle, minimum size=12pt, inner sep=0pt] (b) at (0,0) {$2$};
			\node[draw, circle, minimum size=12pt, inner sep=0pt] (c) at (2,0) {$3$};

			\draw[->] (a) to node[right] {} (b); 
			\draw[->, bend left] (b) to node[left] {} (a); 
			\draw[red,->, bend left] (b) to node[black, above] {$\rho\beta$} (c); 
			\draw[red,->, bend left] (c) to node[black, below] {$\eta \beta$} (b); 

			\draw[->, bend left] (a) to node[right] {} (c); 
			\draw[->] (c) to node[left] {} (a); 
		\end{tikzpicture}
		\par\vspace{0.0ex}
		\centering{(a) $G$}
	\end{minipage}
	\hfill
	\begin{minipage}[b]{0.35\textwidth}
		\raggedright
		\begin{tikzpicture}
			\node[draw, circle, minimum size=12pt, inner sep=0pt] (a) at (1,1.73) {$R_1$};
			\node[draw, circle, minimum size=12pt, inner sep=0pt] (bc) at (1,0) {$R_2$};  
			
			\draw[->, bend left] (a) to node[right] {} (bc); 
			\draw[->, bend left] (bc) to node[left] {} (a); 
		\end{tikzpicture}
		\par\vspace{3.5ex}
		(b) $\underline{G}^-$
	\end{minipage}
	
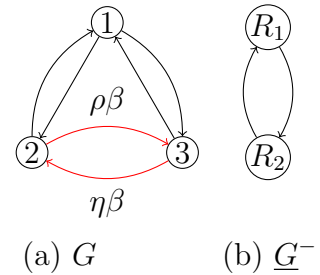
\captionof{figure}{Original graph $G$ and reduced graph $\underline{G}$.}
	\label{fig:graphs_dir}
\end{minipage}

\noindent by $\cev{\gamma}^-_{R_1} = (1/m(1), 0 , 0)^\intercal$ and $\cev{\gamma}^-_{R_2} =  (0, \rho, \eta)/[m(2)\rho + m(3)\eta ]  $. For the masses of the respective reaches we easily find $\underline{m}^-(R_1) = m(a)$ and $\underline{m}^-(R_2) = m(b) + m(c)$. For the edge weights of the graph $\underline{G}^-$, we find:
\begin{align}
	\underline{a}^-(R_1,R_2) &= \underline{m}^-(R_1)\sum\limits_{r \in R_1, s \in R_2} 
	[
	\cev{\gamma}^{-}_{R_1}(r)
	a(r,s)
	\vec{\gamma}_{R_2}^{-}(s)
	] 
	= 
a(1,2) + a(1,3)\\
 	\underline{a}^-(R_2,R_1) &= \underline{m}^-(R_2)\sum\limits_{r \in R_2, s \in R_1} 
 [
 \cev{\gamma}^{-}_{R_1}(r)
 a(r,s)
 \vec{\gamma}_{R_2}^{-}(s)
 ] 
 = \frac{m(2) + m(3)}{m(2)\rho + m(3)\eta}(\rho a(2,1) + \eta a(3,1)).
\end{align}

\noindent 		\begin{minipage}[b]{0.7\textwidth}
	\begin{Ex}\label{counter_intuitive_example} Finally we want to showcase how we may lose connectivity between 
	reaches in the limit $\beta \rightarrow \infty$. Consider again the setting of a three node graph, as in Example \ref{dir_two_edges_div}. However, now assume no incoming edges at node $1$ (i.e.  $a(1,j) = 0$ for $j=1,2$). Furthermore assume no edge from node $1$ to node $2$ (i.e. $a(2,1) = 0$) and no edge from node $3$ to node $2$ (i.e. $\eta = 0$); c.f. Fig. \ref{fig:graphs_dir_II} (a). For this configuration, we find $\underline{a}^-(R_1,R_2) = \underline{a}^-(R_2,R_1) = 0 $. Thus there is no connection between the two nodes in $\underline{G}^-$. We can understand this from the perspective of 'heat flow'; i.e. via the associated 'heat' semigroup $e^{-t \Lin_\beta}$:	
\end{Ex}
	\end{minipage}
	\begin{minipage}[b]{0.4\textwidth}
		\centering
		\begin{minipage}[b]{0.6\textwidth}
			\raggedleft
			\begin{tikzpicture}
				\node[draw, circle, minimum size=12pt, inner sep=0pt] (a) at (1,1.73) {$1$};
				\node[draw, circle, minimum size=12pt, inner sep=0pt] (b) at (0,0) {$2$};
				\node[draw, circle, minimum size=12pt, inner sep=0pt] (c) at (2,0) {$3$};

				\draw[->] (a) to node[right] {} (c); 
				\draw[red,->] (b) to node[black, above] {$\beta$} (c); 
			\end{tikzpicture}
			\par\vspace{.5ex}
			\centering{(a) $G$}
		\end{minipage}
		\hfill
		\begin{minipage}[b]{0.35\textwidth}
			\raggedright
			\begin{tikzpicture}
				\node[draw, circle, minimum size=12pt, inner sep=0pt] (a) at (1,1.73) {$R_1$};
				\node[draw, circle, minimum size=12pt, inner sep=0pt] (bc) at (1,0) {$R_2$};  
			\end{tikzpicture}
			\par\vspace{0ex}
			(b) $\underline{G}^-$
		\end{minipage}
		
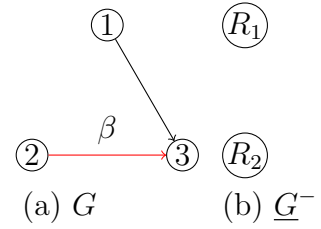
\captionof{figure}{Original graph $G$ and reduced graph $\underline{G}^-$.}
		\label{fig:graphs_dir_II}
	\end{minipage}	

	\noindent	In the setting of Example \ref{counter_intuitive_example}, let us for simplicity set $m(i) = 1$ for all nodes $i = 1,2,3$. With $a(2,1) =1$, $a(3,2) = \beta$ and other edge-weights zero, an explicit calculation shows
		\begin{align}\label{explicit_rep}
	J^-_\uparrow e^{-t\underline{L}^-} J^-_\downarrow &= J^-_\uparrow J^-_\downarrow = \begin{pmatrix}
	1 & 0 & 0 \\
	0 & 1 & 0 \\
	0 & 1 & 0
\end{pmatrix},
\quad	
e^{-t \Lin_\beta}
=
\begin{pmatrix}
	1 & 0 & 0 \\
	0 & 1 & 0 \\
	\frac{1 - e^{-(\beta + 1)t}}{\beta + 1} & -\frac{\beta (e^{-(\beta + 1)t} - 1)}{\beta + 1} & e^{-(\beta + 1)t}
\end{pmatrix}.
		\end{align}
Thus we have the convergence $e^{-t\Lin_\beta} \rightarrow J^-_\uparrow  e^{-t\underline{L}^-} J^-_\downarrow$ as $\beta \rightarrow \infty$.
 The 'heat-flow' from node $1$ to node $3$ at time $t$ is given by $[e^{-t\Lin_\beta}]_{31} = 	(1 - e^{-(\beta + 1)t})/(\beta + 1) = \mathcal{O}(1/\beta)$. As $\beta \gg 1$ increases, the flow from node $1$ to node $3$ becomes increasingly suppressed, until it vanishes at $\beta = \infty$. 
We may intuitively understand this from the perspective of probability flows on the graph $G$: As argued in  \cite[Section 6]{Veerman2020APO}, the 'heat kernel' $e^{-t\Lin}$ generates a probability flow \emph{counter} to directed edges. In the graph $G$, probability may hence flow only \emph{from} node $3$ \emph{to} node $1$ or $2$. In the setting $\beta \gg 1$, the connection to node $2$ is much stronger than to node $1$. Hence most probability mass flows to node $2$, where it gets trapped. By probability mass conservation, this implies less and les probability mass arrives at node $1$, until the connection is completely severed when $\beta = \infty$. This behaviour is reflected in the disconnected  limit graph structure of $\underline{G}^-$.

We note that in the directed setting, the semigroup $t \mapsto e^{-t \Lin_\beta}$ thus behaves somewhat counter-intuitively, and an interpretation in terms of 'heat -flow' becomes challenging. Nevertheless, the fact that $e^{-t\Lin_\beta} \rightarrow J^-_\uparrow  e^{-t\underline{L}^-} J^-_\downarrow$ as $\beta \rightarrow \infty$ in Example \ref{counter_intuitive_example} could be used to shed light on the limit graph structure $\underline{G}^-$. Here we hence establish that this convergence behaviour of the semigroup persists for all finite graphs.\footnote{Our proof may be extended to directed sectorial (c.f. e.g.  \cite{anne2018sectoriality})  graph Laplacians that are also trace-class (or more generally 
Schatten class for finite $p$). 
The author is  
unaware of any
method to extend Theorem \ref{directed_semi_group_convergence} beyond this: Sectoriality is crucial in establishing uniform boundedness of $z \mapsto e^{-z\Lin_\beta} $. This together with a Schatten norm estimate \cite{Bandtlow2003} then allows to apply results of \cite{arendt2001approximation} to translate the resolvent convergence of Theorem \ref{-_conv_theorem} into the semigroup convergence 
of Theorem \ref{directed_semi_group_convergence}. 
}

\begin{Thm}\label{directed_semi_group_convergence}
In the setting of Theorem \ref{-_conv_theorem} with $G$ a finite graph, we have for any $t  > 0$ the convergence 
 $\|e^{-t\Lin_\beta} - J^-_\uparrow  e^{-t\underline{L}^-} J^-_\downarrow\| \rightarrow 0$ as $\beta \rightarrow \infty$.
\end{Thm}
\begin{proof}
Assume we had established that there is a sector $\Sigma_\theta \subset \mathds{C}$ of opening angle $2\theta$ and a constant $M>0$ so that $\|e^{-t\Lin_\beta}\| \leq M$ uniformly in $z \in \Sigma_\theta$ and  $\beta \uparrow \infty$. Then by \cite[Theorem 1.6]{arendt2001approximation} we have that $e^{-z\Lin_\beta}$ converges  towards a  limit operator, with the convergence uniform on compact subsets of 
$\Sigma_\theta $. Denote the limit operator by $T(z)$. By norm-convergence, the operator family $\Sigma_\theta \ni z \mapsto T(z)$ inherits the semigroup property $T(z)T(z') = T(z + z')$ as well as boundedness $\|T(z)\| \leq M, \forall z \in \Sigma_\theta$.
Since we are in finite dimensions, Bolzano Weierstrass guarantees that any sequence $\{t_n\}_{n \in \mathds{N}}$ with $t_{n-1} > t_{n} \rightarrow 0 $ has a strictly decreasing subsequence $\{t_{n_i}\}_{i \in \mathds{N}}$ so that $T(t_{n_i})$ converges. By holomorphicity and the semigroup property, this can be extended to the convergence $\lim_{t \rightarrow 0} T(t) = T(0)$, with a unique limit $T(0)$. By \cite[Section 2]{arendt2001approximation}, $T(0)$ is a projection and $t \mapsto T(t)$ is a $C_0$ semigroup on $T(0)\ell^2(G)$. On the complement of this space,  $T(t)$ vanishes identically: $T(t){\restriction_{(I - T(0))\ell^2(G)}} \equiv 0$. Associated to the semigroup $T(t){\restriction_{T(0)\ell^2(G)}} $ is a generator $A$ whose resolvent at $-\lambda$  may be calculated as
$
	(A + \lambda)^{-1} = \int_0^\infty T(t)e^{-\lambda} dt, \quad \text{for $\Re(\lambda) > 0$}.
$
By dominated convergence for Bochner integrals, we have as $\beta \rightarrow \infty$, that
$
\lim\limits_{\beta \rightarrow \infty}\int_0^\infty e^{-t\Lin_\beta}e^{-\lambda} dt \rightarrow	\int_0^\infty T(t)e^{-\lambda} dt.
$
However, we also have
$
	\lim\limits_{\beta \rightarrow \infty}\int_0^\infty e^{-t\Lin_\beta}e^{-\lambda} dt 
	 = 	\lim\limits_{\beta \rightarrow \infty} (\Lin_\beta - \lambda)^{-1}
	=J^-_\uparrow  (\underline{L}^- - \lambda)^{-1} J^-_\downarrow ,
$
which yields $A  = J^-_\uparrow \underline{L}^- J^-_\downarrow$.
\medskip

		To establish the claim, it thus remains to verify our initial assumption and establish that there is a sector $\Sigma_\theta \subset \mathds{C}$ of opening angle $2\theta$ and a constant $M>0$ so that $\|e^{-t\Lin_\beta}\| \leq M$ uniformly in $z \in \Sigma_\theta$ and  $\delta \downarrow 0$. We first note that upon setting $y = z/\beta$ and $\delta = 1/\beta$, we have the following equivalence:
	\begin{align}
		\left\|e^{-z(\Lin_{E\setminus \tilde{E}} + \beta \Lin_{\tilde{E}})}\right\| \leq M\ (\forall z \in \Sigma_\theta,\ \beta \uparrow \infty)
		\quad \Leftrightarrow \quad 
		\left\|e^{-y(\Lin_{\tilde{E}} + \delta \Lin_{E\setminus \tilde{E}})}\right\| \leq M\ (\forall y \in \Sigma_\theta,\ \delta \downarrow 0)
	\end{align}

	All eigenvalues of $\Lin_{\tilde{E}}$ and $\Lin_{\tilde{E}} + \delta \Lin_{E\setminus \tilde{E}}$ have non-negative real part \cite{Veerman2020APO}. By finiteness, there thus exists $0< \phi < \pi/2$ so that all eigenvalues $\lambda \in \sigma(\Lin_{\tilde{E}})$ are contained in the interior of the wedge  $\mathring{\Sigma}_\phi \supseteq \sigma(\Lin_{\tilde{E}})$.  Since roots of (characteristic) polynomials depend continuously on polynomial coefficients,  eventually  also  all eigenvalues of $\Lin_{\tilde{E}} + \delta \Lin_{E\setminus \tilde{E}}$ lie within $\mathring{\Sigma}_\phi$; as soon as $\delta \downarrow 0$ is sufficiently small. 
	Using contour integration, we have
	\begin{align}
			e^{-y(\Lin_{\tilde{E}} + \delta \Lin_{E\setminus \tilde{E}})} = P^-_{\Lin_{\tilde{E}} + \delta \Lin_{E\setminus \tilde{E}}} + \frac{1}{2 \pi i}\oint_\Gamma e^{-y\zeta}\left(\Lin_{\tilde{E}} + \delta \Lin_{E\setminus \tilde{E}} - \zeta\right)^{-1} d\zeta
	\end{align}
with $P^-_{\Lin_{\tilde{E}} + \delta \Lin_{E\setminus \tilde{E}}}$ the Riesz projection onto $0 \in \sigma(\Lin_{\tilde{E}} + \delta \Lin_{E\setminus \tilde{E}})$ and the path $\Gamma$ chosen as in Fig. \ref{integration_path}. By analytic perturbation theory  \cite[Chapter II]{kato1995perturbation}
	we find that $P^-_{\Lin_{\tilde{E}} + \delta \Lin_{E\setminus \tilde{E}}}$ depends continuously on $\delta$, and we have $P^-_{\Lin_{\tilde{E}} + \delta \Lin_{E\setminus \tilde{E}}} \rightarrow P^-_{\Lin_{\tilde{E}}}$ as $\delta \downarrow 0$.

		\noindent		\begin{minipage}[b]{0.7\textwidth}
 Thus for any $\epsilon > 0$  we find $\|P^-_{\Lin_{\tilde{E}} + \delta \Lin_{E\setminus \tilde{E}}}\| \leq  \|P^-_{\Lin_{\tilde{E}} }\|  + \epsilon$ for $\delta \ll 1$ sufficiently small.  Again by continuous dependence of eigenvalues on matrix entries, we note that there exists an $ \epsilon > 0$ so that 
$
\dist(\Gamma, \sigma(\Lin_{\tilde{E}} + \delta \Lin_{E\setminus \tilde{E}})) 
\geq \epsilon
$ uniformly in $\delta \downarrow 0$ (c.f. Fig \ref{integration_path}).
Thus by \cite[Theorem 4.1]{Bandtlow2003},  for
$\zeta \in \Gamma$:

\begin{align}
&\left\| (  \Lin_{\tilde{E}} + \delta \Lin_{E\setminus \tilde{E}}     - \zeta)^{-1} \right\| \\
	\leq& 
	\frac{1}{\dist(\zeta, \sigma(\Lin_{\tilde{E}} + \delta \Lin_{E\setminus \tilde{E}}))} \exp\left(  \frac{\|\Lin_{\tilde{E}} + \delta \Lin_{E\setminus \tilde{E}}\|_1}{\dist(\zeta, \sigma(\Lin_{\tilde{E}} + \delta \Lin_{E\setminus \tilde{E}}))} \right)\\
	\leq&
	e^{(\|\Lin_{\tilde{E}}\|_1 + \epsilon)/\epsilon}/\epsilon.
\end{align}

With this we find
\begin{align}
	\left\| \frac{1}{2 \pi i}\oint_\Gamma e^{-y\zeta}\left(\Lin_{\tilde{E}} + \delta \Lin_{E\setminus \tilde{E}} - \zeta\right)^{-1} d\zeta\right\|
	\leq 
	\frac{e^{\|\Lin_{\tilde{E}}\|_1/\epsilon + 1}}{2\pi \epsilon} \oint_\Gamma |e^{-y\zeta}|d\zeta.
\end{align}
	Next we note that there is a sector $\Sigma_\alpha \subseteq \mathds{C}$ so that the entire path is contained within this sector. Upon setting $\theta = \pi/2 - \alpha $, we have that for any $y \in \Sigma_\phi$ and $\zeta \in \Gamma \subseteq \Sigma_\alpha$ the argument $\arg(y\zeta)$ satisfies $\arg(y\zeta) \leq |\arg(y)| + |\arg(\zeta)| < \alpha + (\pi/2 - \alpha) < \pi/2 $. Thus $\Re(y\zeta) > 0$ and hence $|e^{-y\zeta}| \leq 1$. 
	\end{minipage}
	\begin{minipage}[b]{0.4\textwidth}
		\centering

\begin{tikzpicture}[scale=2]

	\draw[->] (-0.15,0) -- (2.05,0) node[right] {Re};
	\draw[->] (0,-2) -- (0,2) node[above] {Im};

	\filldraw (0,0) circle (0.03) node[anchor=north east] {0};

	\def\thetaarg{22.5} 
	\def\alphaarg{45} 
	\pgfmathsetmacro{\r}{2 / cos(\thetaarg)}
	\pgfmathsetmacro{\rp}{2 / cos(\alphaarg)}
	\pgfmathsetmacro{\s}{1.0}

	\coordinate (A) at ({\r*cos(\thetaarg)},{\r*sin(\thetaarg)});
	\coordinate (B) at ({\r*cos(-\thetaarg)},{\r*sin(-\thetaarg)});

	\coordinate (C) at ({\rp*cos(\alphaarg)},{\rp*sin(\alphaarg)});
	\coordinate (D) at ({\rp*cos(-\alphaarg)},{\rp*sin(-\alphaarg)});

	\draw[blue, thick] (0,0) -- (A);
	\draw[blue, thick] (0,0) -- (B);

	\draw[red, thick] (0,0) -- (C);
	\draw[red, thick] (0,0) -- (D);

	\draw[blue, thick] (0.6,0) arc (0:\thetaarg:0.6);
	\node at ({0.5*cos(0.5*\thetaarg)}, {0.5*sin(0.5*\thetaarg)}) {$\phi$};

	\draw[red, thick] (0.6,0) arc (0:-\alphaarg:0.6);
	\node at ({0.5*cos(-0.5*\alphaarg)}, {0.5*sin(-0.5*\alphaarg)}) {$\alpha$};

	\node[blue] at ({2.3*cos(0.85*\thetaarg)}, {2.3*sin(0.85*\thetaarg)}) {$\Sigma_{\phi}$};
	\node[red] at ({2.3*cos(0.85*\alphaarg)}, {2.3*sin(0.85*\alphaarg)}) {$\Sigma_{\alpha}$};

	\coordinate (P) at (0.8,0.5);
	\coordinate (Pprime) at (0.8,-0.5);

	\pgfmathsetmacro{\stepX}{\s*cos(\thetaarg)}
	\pgfmathsetmacro{\stepY}{\s*sin(\thetaarg)}
	\pgfmathsetmacro{\stepYprime}{-\s*sin(\thetaarg)}

	\coordinate (Q) at ({0.8+\stepX},{0.5+\stepY});
	\coordinate (Qprime) at ({0.8+\stepX},{-0.5+\stepYprime});

	\node at ($(Q)!0.75!(Qprime) + (0.2,0)$) {$\Gamma$};

	\draw[black, thick,
	postaction={decorate},
	decoration={markings, mark=at position 0.5 with {\arrow{>}}}
	] (Q) -- (P);

	\draw[black, thick,
	postaction={decorate},
	decoration={markings, mark=at position 0.5 with {\arrow{>}}}
	] (P) -- (Pprime);

	\draw[black, thick,
	postaction={decorate},
	decoration={markings, mark=at position 0.5 with {\arrow{>}}}
	] (Pprime) -- (Qprime);

	\draw[black, thick,
	postaction={decorate},
	decoration={markings, mark=at position 0.5 with {\arrow{>}}}
	] (Qprime) -- (Q);

	\filldraw[black] (1.2, 0.25) circle (0.03);
	\filldraw[black] (1.5, 0.15) circle (0.03);
	\filldraw[black] (1.2, -0.25) circle (0.03);
	\filldraw[black] (1.5, -0.15) circle (0.03);

	\node at (1.2, 0.36) {$\lambda_2$};
	\node at (1.5, 0.26) {$\lambda_3$};
	\node at (1.2, -0.36) {${\lambda}_4$};
	\node at (1.5, -0.26) {${\lambda}_5$};

	\draw[black, thick] (0,0.6) arc (90:\alphaarg:0.6);
	\node at ({0.15*cos(0.5*\alphaarg)}, {0.15*sin(0.5*\alphaarg) +0.3}) {$\theta$};

\end{tikzpicture}

		
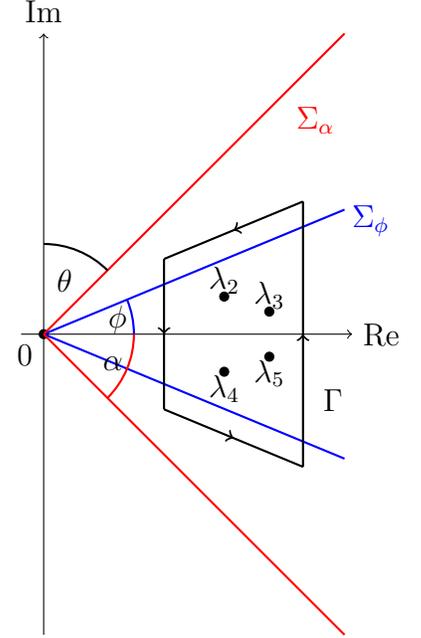
\captionof{figure}{Path $\Gamma$ around eigenvalues 
			$ \sigma(L^-_{ \tilde{E}} + \delta L^-_{E\setminus \tilde{E}}  )\setminus \{0\}$.}
		\label{integration_path}
	\end{minipage}

 \noindent In total we  have thus indeed found that as desired $\forall z \in \Sigma_\theta,\ \beta \uparrow \infty$
\begin{align}
	\left\|e^{-z(\Lin_{E\setminus \tilde{E}} + \beta \Lin_{\tilde{E}})}\right\|  \leq \left(\|P^-_{\Lin_{\tilde{E}} + \delta \Lin_{E\setminus \tilde{E}}}\|  + \epsilon +  \frac{e^{\|\Lin_{\tilde{E}}\|_1/\epsilon + 1}}{2\pi \epsilon} \oint_\Gamma 1 d\zeta \right)=: M.
\end{align} 
\end{proof}

\subsubsection{Out-degree Laplacian} Having understood the in-degree setting, let us now consider the out-degree Laplacian. As above, we first specify the coarsified  graph $\underline{G}^+$:

\begin{Def}\label{deg_coarse_G_def} Given a graph $G = (V, m, E, a)$ and a countable collection $\{\tilde{E}_i\}_{i \in I}$ of clustered edges,
	we define a reduced graph $\underline{G}^+ = \{\underline{V}^+, \underline{m}^+, \underline{E}^+,\underline{a}\}^+$ as follows:
	Define an equivalence relation "$\sim_{\tilde{E}^\intercal}$" on the node-set $V$ by setting $a \sim_{\hat{E}^\intercal} b $ iff
	$a$ and $b$ are contained in the same reach of $(V, \tilde{E}^\intercal)$.
	\begin{enumerate}[label=(\roman*), topsep=0pt, itemsep=-.2em]
		\item 	The reduced node set $\underline{V}^+$, is defined as the set of reaches of ($V, \tilde{E}^\intercal$):
		\begin{align}
			\underline{V}^+ = V/ \sim_{\tilde{E}^\intercal}.
		\end{align}
		\item The weight of a node $S^\intercal \in \underline{V}^+$ is given as
		\begin{align}\label{out_deg_mass}
			\underline{m}^+(S^\intercal) := 	\langle  \chi_{S^\intercal} ,P^+  \chi_{S^\intercal}  \rangle_{\ell^2(G)}
		\end{align}
		\item The edge-weight function $\underline{a}:\underline{V} \times \underline{V} \rightarrow [0, \infty)$ is defined similarly as:
		\begin{align}\label{weight_def_+}
			\underline{a}^+(R^\intercal, S^\intercal) := \left(\sum\limits_{r \in R^\intercal, s \in S^\intercal} 
			[
			\cev{\gamma}^{+}_{R^\intercal}(r)
			a(r,s)
			\vec{\gamma}_{S^\intercal}^{+}(s)
			]\right) \cdot \underline{m}^+(S^\intercal)
		\end{align}
		\item The set of edges $\underline{E}^+$ is given as the support of $\underline{a}^+$ on $\underline{V}^+ \times \underline{V}^+$:
		\begin{equation}
			\underline{E}^+ = \{(R,S)\in\underline{V}\times\underline{V}: \underline{a}^+(R,S) >0 \}.
		\end{equation}
	\end{enumerate}
\end{Def}

\medskip

Also in the out-degree setting, the notion of 'mass associated to a reach' (\ref{out_deg_mass}) reduces to the definition of (\ref{mass_def_undirected}) in the undirected setting. Furthermore, since $\langle \chi_{S^\intercal}, \vec{\gamma}^{+}_{S^{\intercal}}\rangle =1$:
\begin{align}
	\underline{m}^+(S^\intercal) = \langle \cev\gamma^{+}_{S^{\intercal}}, \chi_{S^\intercal} \rangle_{\ell^2(G)}
	=
	\sum\limits_{s \in S} \cev\gamma^{+}_{S^\intercal}(s) m(s) = \sum\limits_{s \in H(S^\intercal)} m(s) + \sum\limits_{s \in C(S^\intercal)} \cev\gamma^{+}_{S}(s) m(s).
\end{align}
With this, we find in the finite mass setting exactly as for the in-degree setting, that
\begin{equation}
	\underline{M}^+ = \sum\limits_{S^\intercal \in \underline{V}} \underline{m}^+(S^\intercal)
	=
	\sum\limits_{S^\intercal \in \underline{V}}
	\left(
	\sum\limits_{s \in H(S^\intercal)} m(s) + \sum\limits_{s \in C(S^\intercal)} \vec\gamma^{+}{S^\intercal}(s) m(s)
	\right)
	= \sum\limits_{s \in V} m(s)
	= M.
\end{equation}

Regarding Definition (\ref{weight_def_+}) pertaining to aggregated edge-weigths, we  note that --- as in the in-degree setting --- our definition exactly reduces to our earlier definition (\ref{edge_def_undir}): In the undirected setting, we have $	\vec{\gamma}^{+}_{S^\intercal} = \chi_{S^\intercal}/(\sum_{v \in R} m(r))$ and $	\cev{\gamma}_{S^\intercal}^{+}= \chi_{S^\intercal}$, and $\underline{m}^{+}(S^\intercal) = \sum_{v \in S^\intercal} m(s)$. Thus
\begin{align}
	\underline{a^+}(R^\intercal, S^\intercal) 
	= 	\left(\sum\limits_{r \in R^\intercal, s \in S^\intercal} 
	\chi_{R^\intercal}(r)
	a(r,s)
	\chi_{S^\intercal}(s) \right) \cdot \frac{\underline{m}^{+}(S^\intercal)}{\underline{m}^{+}(S^\intercal)}
	= 
	\sum\limits_{r \in R, s \in S} 
	a(r,s).
\end{align}

		\medskip

	Again, we we now decompose the Riesz projector onto $0 \in \sigma(\Lout_{\tilde{E}})$ as $P^+ = J^+_\uparrow J^+_\downarrow$. Here $J^+_\downarrow$ maps to  $\ell^2(\underline{G}^+)$, and $J^+_\uparrow$ maps in the opposite direction; back to $\ell^2(G)$.
	\begin{Def}\label{proj_interp_outdeg}
		At the level of Hilbert spaces, we define the bounded projection $J^+_\downarrow: \ell^2(G) \rightarrow \ell^2(\underline{G}^+)$ and interpolation $J^+_\uparrow: \ell^2(\underline{G}^+) \rightarrow \ell^2(G)$ operators as 
		\begin{align}
			[J^+_\downarrow f](R^\intercal) = \langle \cev\gamma^{+}_{R^\intercal} , f \rangle_{\ell^2(V)}/m^+(R^\intercal),\quad \text{and} \quad J^+_\uparrow \underline{f} = \sum_{R^\intercal \in \mathfrak{R}^\intercal} \underline{m}^+(R^\intercal)  \underline{f}(R^\intercal)  \vec\gamma^{+}_{R^\intercal}.
		\end{align}
	\end{Def}

\begin{Rem}	
The masses in
Definition \ref{proj_interp_outdeg} above
  are placed exactly so that  $J_\uparrow^+, J_\downarrow^+$ preserve probability distributions in the following sense:
	\begin{align}
		\sum\limits_{R^\intercal \in \underline{V}^+} [J^+_\downarrow f(R^\intercal)] \underline{m}^+(R^\intercal) 
		=
				\sum\limits_{R^\intercal \in \underline{V}^+}	\langle \cev\gamma^{+}_{R^\intercal} , f \rangle_{\ell^2(V)} \frac{\underline{m}^+(R^\intercal)}{ \underline{m}^+(R^\intercal) } 
		=
			\sum\limits_{R^\intercal \in \underline{V}^+}	\langle \cev\gamma^{+}_{R^\intercal} , f \rangle_{\ell^2(V)}.
		\end{align}
But since the family $\{\cev\gamma^{+}_{R^\intercal}\}_{R^\intercal \in \mathfrak{R}^\intercal}$ forms a partition of unity on $\ell^2(G)$ (i.e. $\sum_{R^\intercal \in \underline{V}^+} \cev\gamma^{+}_{R^\intercal}(j) = 1 $), we have '$\sum\limits_{R^\intercal \in \underline{V}^+}	 \cev\gamma^{+}_{R^\intercal} = \chi_G$', and hence $\sum\limits_{R^\intercal \in \underline{V}^+}	\langle \cev\gamma^{+}_{R^\intercal} , f \rangle_{\ell^2(V)} 
= \sum\limits_{v\in V} f(v) m(v)$.

Similarly, if $\underline{f} \in \ell^2(\underline{G}^-) \cap \ell^1(\underline{G}^-)$ is a probability distribution, we have
\begin{align}
	\sum_{v \in V} [J^+_\uparrow \underline{f}](v) m(v) =&  	\sum_{v \in V} \langle e_v , [J^+_\uparrow \underline{f}]\rangle_{\ell^2(G)}
	=  	\sum_{v \in V}
	 \langle e_v ,	
	\sum_{R^\intercal \in \mathfrak{R}^\intercal} \underline{m}^+(R^\intercal)  \underline{f}(R^\intercal)  \vec\gamma^{+}_{R^\intercal}
	\rangle_{\ell^2(G)}\\
	=& \sum_{R^\intercal \in \mathfrak{R}^\intercal} \underline{m}^+(R^\intercal)      \underline{f}(R^\intercal)         \left( \sum_{v \in V} \langle e_v ,	 \vec\gamma^{+}_{R^\intercal}
	\rangle_{\ell^2(G)} \right).
\end{align}
By Proposition 
\ref{right_kernel_+}  $\sum_{j \in V}\, \langle e_j, \vec\gamma^{+}_{R^\intercal}\rangle =1$, so that the factor in brackets is equal to unity. 
\end{Rem}

	To the graph $\underline{G}^+$ of Definition \ref{deg_coarse_G_def} we can then canonically associate a graph Lapalcian $\underline{L}^+$ via  Definition \ref{laplacian_definition}. As we show now, this is the limit Laplacian, to which the family $\Lout_\beta = \Lout_{E \setminus \tilde{E}} + \beta \Lout_{\tilde{E}}$ converges:
	\begin{Thm}\label{+_conv_theorem}
		We have
$
			\| (\Lout_\beta - z)^{-1} - J^+_{\uparrow} (\underline{L}^+ -z)^{-1}  J^+_\downarrow \|
			\lesssim 1/\beta \xrightarrow{\beta \rightarrow \infty} 0.
$
	\end{Thm}
	In order to prove this result we first again establish an auxiliary Lemma:
	\begin{Lem}\label{aux_lemma_+}
		Let $\Lout_{\hat{E}}$	be an out-degree Laplacian with finite range. Denote by $P^+ = J^+_\uparrow J^+_\downarrow$ the Riesz projection onto $0 \in \sigma(\Lout_{\hat{E}})$. We have
$
			\|  (\beta \Lout_{\tilde{E}} - z)^{-1} - P^+/(-z)  \| \lesssim 1/\beta.
	$
	\end{Lem}
	\begin{proof} The proof proceeds in complete analogy to that of Lemma \ref{aux_lemma_-}.
	\end{proof}

	\begin{proof}[Proof of Theorem \ref{-_conv_theorem}]
		Let us use the notation  $\Lout_{G, \beta} =  \Lout_{E \setminus \tilde{E}} +\beta \Lout_{\tilde{E}}$ and $\Lin_{G^\intercal, \beta} =  \Lin_{E^\intercal  \setminus \tilde{E}^\intercal } +\beta \Lin_{\tilde{E}^\intercal }$. Furthermore note that for $P^+$ the Riesz projector onto $0 \in \sigma(\Lout_{\tilde{E}})$
		and $P^-$ the Riesz projector onto $0 \in \sigma(\Lin_{\tilde{E}^\intercal})$, we have the adjoint-ness relation $P^- = [P^+]^\star$. This follows from the defining equation (\ref{riesz_projector}) together with Proposition 
		\ref{in_out_transpose_prop}.
		
By Proposition \ref{in_out_transpose_prop}, we then have $
	[(\Lout_{G, \beta} - z )^{-1}]^\star = (\Lin_{G^\intercal,\beta } - \overline{z})^{-1}$, and by the proof of Theorem \ref{-_conv_theorem}, we have
	$\|(\Lin_{G^\intercal,\beta } - \overline{z})^{-1} -  P^-  ( P^-  \Lin_{G^\intercal,\beta } P^- - \overline{z})^{-1}     P^-     \|  \lesssim 1/\beta$. Since $\|\cdot\|$ is invariant under taking adjoints, also $\|(\Lout_{G,\beta } - z)^{-1} -  P^+  ( P^+  \Lout_{G^+,\beta } P^+ - z)^{-1}     P^+     \|  \lesssim 1/\beta$.

Since $P  \Lout_{G,\beta } P = P\Lout_{E \setminus \tilde{E}} P $, it remains to establish that
$P\Lout_{E \setminus \tilde{E}} P = J^+_{\uparrow} \underline{L}^+ J^+_{\downarrow}$. 	
This follows after noting $\underline{L}^+ = J^+_{\downarrow} \Lout_{E\setminus \tilde{E}}  J^+_\uparrow$. Indeed:
\begin{align}
\langle e_{R^\intercal}, J^+_{\downarrow} \Lout_{E\setminus \tilde{E}}  J^+_\uparrow e_{S^\intercal}\rangle_{\ell^2(\underline{G}^+)} = \langle e_{R^\intercal}, e_{R^\intercal}\rangle_{\ell^2(\underline{G}^+)}\cdot
\langle\cev{\gamma}^+_{R^\intercal},  \Lout_{E\setminus \tilde{E}} \vec{\gamma}^+_{S^\intercal} \rangle_{\ell^2(G)}.
\end{align}
We have $ \langle e_{R^\intercal}, e_{R^\intercal}\rangle_{\ell^2(\underline{G}^+)} =  \underline{m}^+(R^\intercal) $.  
For $\langle\cev{\gamma}^+_{R^\intercal},  \Lout_{E\setminus \tilde{E}} \vec{\gamma}^+_{S^\intercal} \rangle_{\ell^2(G)}$, we find that if $R^\intercal \neq S^\intercal$:
\begin{align}
	\langle\cev{\gamma}^+_{R^\intercal},  \Lout_{E\setminus \tilde{E}} \vec{\gamma}^+_{S^\intercal} \rangle_{\ell^2(G)} 
	= - 	\frac{\underline{m}^+(S^\intercal)}{\underline{m}^+(R^\intercal)}\sum\limits_{r \in R^\intercal, s \in S^\intercal} 
	[
	\cev{\gamma}^{+}_{R^\intercal}(r)
	a(r,s)
	\vec{\gamma}_{S^\intercal}^{+}(s)
	].
\end{align}
Thus the off-diagonal elements of $J^+_{\downarrow} \Lout_{E\setminus \tilde{E}}  J^+_\uparrow $ agree with those of $\underline{\Lout}$ as associated via Definition \ref{laplacian_definition} to the graph $\underline{G}^+$ determined by Definition \ref{indeg_coarse_G_def}.

To see that also the diagonal entries of the two operators agree, we note that for fixed $R^\intercal$, we have (since by Proposition \ref{prop:leftkernel_+} the $\{\cev{\gamma}_{R^\intercal}\}_{R^\intercal \in \mathfrak{R}}$ are a  partition of unity)
\begin{align}
	\sum_{R^\intercal \in \underline{V}} \langle e_{R^\intercal}, J^+_{\downarrow} \Lout_{E\setminus \tilde{E}}  J^+_\uparrow e_{S^\intercal}\rangle_{\ell^2(\underline{G}^+)} 
	=   \sum_{R^\intercal \in \underline{V}} \langle\cev{\gamma}^+_{R^\intercal},  \Lout_{E\setminus \tilde{E}} \vec{\gamma}^+_{S^\intercal} \rangle_{\ell^2(G)}  \underline{m}^+(S^\intercal)
	= 0
\end{align}
Thus again, we indeed find $J^+_{\downarrow} \Lout_{E\setminus \tilde{E}}  J^+_\uparrow = \underline{L}^+$, since
\begin{align}
	-\langle e_{S^\intercal}, J^+_{\downarrow} \Lout_{E\setminus \tilde{E}}  J^+_\uparrow e_{S^\intercal}\rangle_{\ell^2(\underline{G}^+)}  
	= \sum_{\substack{R^\intercal \in \underline{V} \\ R^\intercal \neq S^\intercal}} \langle e_{R^\intercal}, J^+_{\downarrow} \Lout_{E\setminus \tilde{E}}  J^+_\uparrow e_{S^\intercal}\rangle_{\ell^2(\underline{G}^+)}  
	= \sum_{\substack{R^\intercal \in \underline{V} \\ R^\intercal \neq S^\intercal}} \langle e_{R^\intercal}, \underline{L}^+ e_{S^\intercal}\rangle_{\ell^2(\underline{G}^+)}.
\end{align}
	\end{proof}

	From the proof of the above theorem, we immediately find again as in in the in-degree setting that 
	the Laplacian on $\tilde{E}$ need not diverge linearly and find:
	
	\begin{Cor}
		Instead of linear scaling, consider a modification $a \mapsto a_\beta$ on $\tilde{E}$ so that for some projection $P$ onto the kernel of a fixed auxiliary Laplacian on $(V,\tilde{E})$. Then  the  Laplacian $\Lout_{\tilde{E}, \beta}$ on $\tilde{E}$ satisfies $\| (\Lout_{\tilde{E},\beta} -z)^{-1} - P^+/(-z)\| \rightarrow 0$. Then
		\begin{align}
			\| (\Lout_\beta - z)^{-1} - J^-_{\uparrow} (\underline{L}^+ -z)^{-1}  J^+_\downarrow \|\leq C \| (\Lin_{\tilde{E},\beta} -z)^{-1} - P^+/(-z)\| \xrightarrow{\beta \rightarrow \infty} 0.
		\end{align}
	\end{Cor}
\noindent The same caveat on generality of this result as in the in-degree setting applies.
	
	\medskip

\noindent		\begin{minipage}[b]{0.7\textwidth}
	\begin{Ex}\label{dir_two_edges_div_out}	Let us revisit the setting of Examples \ref{one_edge_div} \& \ref{dir_two_edges_div}:
		We again consider an original graph $G$ on three nodes. 
		The out-degree Laplacian $\beta \Lout_{\tilde{E}}$ is then given as 
		\begin{align}
			\beta	\Lout_{\tilde{E}} =   \beta  M^{-1}		\begin{pmatrix}
				0 & 0 & 0 \\
				0& \rho & -\eta \\
				0 & -\rho & \eta
			\end{pmatrix}.
		\end{align}
		The right kernel of $	\Lout_{\tilde{E}} $ is spanned by  $ \vec{\gamma}^{+}_{R^\intercal_1} = (1/m(1),0,0)^\intercal$, $\vec{\gamma}^{+}_{R^\intercal_2} =  (0,\eta,\rho)^\intercal/[m(2) \eta + m(3) \rho]$. A basis of the left kernel is given by $\cev{\gamma}^+_{R\intercal_1} = (1, 0 , 0)^\intercal$  and $\cev{\gamma}^+_{R^\intercal_2} = (0, 1 , 1)^\intercal$. \noindent For the masses of the respective reaches we find $\underline{m}^-(R^\intercal_1) = m(a)$ and $\underline{m}^-(R^\intercal_2) = m(b) + m(c)$. 
	\end{Ex}
\end{minipage}
\begin{minipage}[b]{0.4\textwidth}
	\centering
	\begin{minipage}[b]{0.6\textwidth}
		\raggedleft
		\begin{tikzpicture}
			\node[draw, circle, minimum size=12pt, inner sep=0pt] (a) at (1,1.73) {$1$};
			\node[draw, circle, minimum size=12pt, inner sep=0pt] (b) at (0,0) {$2$};
			\node[draw, circle, minimum size=12pt, inner sep=0pt] (c) at (2,0) {$3$};

			\draw[->] (a) to node[right] {} (b); 
			\draw[->, bend left] (b) to node[left] {} (a); 
			\draw[red,->, bend right] (b) to node[black, below] {$\rho\beta$} (c); 
			\draw[red,->, bend right] (c) to node[black, above] {$\eta \beta$} (b); 
			\draw[->, bend left] (a) to node[right] {} (c); 
			\draw[->] (c) to node[left] {} (a); 
		\end{tikzpicture}
		\par\vspace{0.0ex}
		\centering{(a) $G^\intercal$}
	\end{minipage}
	\hfill
	\begin{minipage}[b]{0.35\textwidth}
		\raggedright
		\begin{tikzpicture}
			\node[draw, circle, minimum size=12pt, inner sep=0pt] (a) at (1,1.73) {$R_1$};
			\node[draw, circle, minimum size=12pt, inner sep=0pt] (bc) at (1,0) {$R_2$};  
			
			\draw[->, bend left] (a) to node[right] {} (bc); 
			\draw[->, bend left] (bc) to node[left] {} (a); 
		\end{tikzpicture}
		\par\vspace{3.5ex}
		(b) $\underline{G}^-$
	\end{minipage}
	
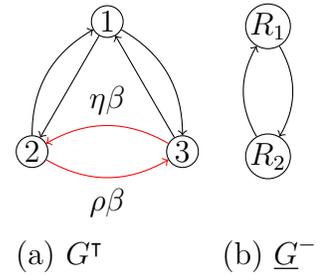
\captionof{figure}{Graph $G$ transposed, and reduced graph $\underline{G}^+$.}
	\label{fig:graphs_dir_gen}
\end{minipage}
\noindent For the edge weights of the graph $\underline{G}^+$, we find:
\begin{align}
	\underline{a}^+(R^\intercal_1, R^\intercal_2) :=& \sum\limits_{r \in R^\intercal_1, s \in R^\intercal_2} 
\cev{\gamma}^{+}_{R^\intercal_1}(r)
a(r,s)
\vec{\gamma}_{R^\intercal_2}^{+}(s)
 \cdot \underline{m}^+(R^\intercal_2)\\
  =& (a(1,2) \eta + a(1,3) \rho) \frac{m(2) + m(3)}{\eta m(2) + \rho m(3)},\\
  	\underline{a}^+(R^\intercal_2, R^\intercal_1) :=& \sum\limits_{r \in R^\intercal_2, s \in R^\intercal_1} 
  \cev{\gamma}^{+}_{R^\intercal_2}(r)
  a(r,s)
  \vec{\gamma}_{R^\intercal_1}^{+}(s)
  = a(2,1) + a(3,1).
\end{align}
Comparing 
 with
Example \ref{dir_two_edges_div}, we find that generically $\underline{G}^- \neq \underline{G}^+$.

\medskip

As in the in-degree setting, one finally establishes:
\begin{Thm}\label{directed_semi_group_convergence_+}
	In the setting of Theorem \ref{+_conv_theorem} with $G$ a finite graph, we have for any $t  > 0$ the convergence 
	$\|e^{-t\Lout_\beta} - J^+_\uparrow  e^{-t\underline{L}^+} J^+_\downarrow\| \rightarrow 0$ as $\beta \rightarrow \infty$.
\end{Thm}

	\bibliographystyle{unsrt}
	\bibliography{bibfile}
\end{document}